\documentclass[11pt,pdftex, reqno]{amsart}
\pdfoutput=1

\usepackage{amsmath, amsthm}
\usepackage{eucal}

\usepackage{palatino}
\usepackage{euler}

\usepackage{amssymb}
\usepackage{amscd}
\usepackage{latexsym}
\usepackage{epsfig}
\usepackage{graphicx, xcolor}
\usepackage{amsfonts}
\usepackage{psfrag}
\usepackage{caption}
\usepackage{setspace}
\usepackage{mathabx}

\usepackage{fullpage}
\usepackage{draftcopy}
\usepackage{tikz}

\usepackage[normalem]{ulem}
\usepackage{pifont}

\usepackage{verbatim}

\usepackage{thmtools}
\usepackage{thm-restate} 

\usepackage{tabmac}
\usepackage[pdfpagelabels]{hyperref}

\usepackage{dsfont}

\usepackage{url}
\usepackage{cite}

\usepackage{dsfont}

\usepackage[margin=1in, headsep=.25in]{geometry}

\input xy
\xyoption{all}
\UseComputerModernTips

\numberwithin{equation}{section}
\numberwithin{figure}{section}

\newtheorem{theorem}{Theorem}[section]

\newtheorem{lemma}[theorem]{Lemma}
\newtheorem{proposition}[theorem]{Proposition}
\newtheorem{corollary}[theorem]{Corollary}

\newtheorem{alg}[theorem]{Algorithm}

\theoremstyle{definition}
\newtheorem{definition}[theorem]{Definition}
\newtheorem{remark}[theorem]{Remark}
\newtheorem{example}[theorem]{Example}

\newcommand{\C}{{\mathbb{C}}}

\newcommand{\Z}{{\mathbb{Z}}}

\newcommand{\N}{{\mathbb{N}}}

\newcommand{\bfa}{\mathbf{a}}

\renewcommand{\t}{\mathfrak{t}}

\newcommand{\wt}{\operatorname{wt}}

\newcommand{\sh}{\operatorname{sh}}

\newcommand{\lift}{\text{lift}}

\newcommand{\row}{\operatorname{row}}

\definecolor{limegreen}{rgb}{0.2, 0.8, 0.2}
\definecolor{amethyst}{rgb}{0.6, 0.4, 0.8}
\definecolor{deeppink}{rgb}{1.0, 0.08, 0.58}

\newcommand{\red}{\textcolor{red}}

\newcommand{\green}{\textcolor{limegreen}}

\usepackage{freetikz}
\usepackage{scalerel}
\begin{document}

\pagestyle{headings}

\title{Crystal Chute Moves on Pipe Dreams}

\author{Sarah Gold}
\address{Sarah Gold, Upper School Mathematics, Friends' Central School, Wynnewood, PA 19096}
\email{sgold@friendscentral.org}
\author{Elizabeth Mili\'cevi\'c}
\address{Elizabeth Mili\'cevi\'c, Department of Mathematics, Haverford College, Haverford, PA 19041}
\email{emilicevic@haverford.edu}
\author{Yuxuan Sun}
\address{Yuxuan Sun, School of Mathematics, University of Minnesota Twin Cities, Minneapolis, MN 55455}
\email{sun00816@umn.edu}

\thanks{The authors were partially supported by NSF Grant DMS-2202017.}

\begin{abstract}
Schubert polynomials represent a basis for the cohomology of the complete flag variety and thus play a central role in geometry and combinatorics. In this context, Schubert polynomials are generating functions over various combinatorial objects, such as rc-graphs or reduced pipe dreams. By restricting Bergeron and Billey’s chute moves on rc-graphs, we define a Demazure crystal structure on the monomials of a Schubert polynomial.  As a consequence, we provide a method for decomposing Schubert polynomials as sums of key polynomials, complementing related work of Assaf and Schilling via reduced factorizations with cutoff, as well as Lenart’s coplactic operators on biwords. 
\end{abstract}

\maketitle



\section{Introduction}\label{sec:intro}

Schubert polynomials are fundamental objects which lie at the intersection of geometry, representation theory, and algebraic combinatorics. By a classical theorem of Borel, the cohomology of the manifold of complete flags in $\C^n$ with integer coefficients is canonically isomorphic to the quotient of $\Z[x_1, \dots, x_n]$ by the ideal generated by the symmetric polynomials without constant term \cite{Borel}. The geometry of the flag variety is best captured by the cohomology classes of the Schubert varieties, which correspond to Schubert polynomials under Borel's isomorphism, generalizing the role of the Schur polynomials in the cohomology of the Grassmannian. In addition to encoding geometric information about the flag variety, individual Schubert polynomials also exhibit rich combinatorial and representation theoretic structures, as developed in \cite{LS-NCschuberts, KohnertThesis, reiner_key_1995, lenart_unified_2004, assaf_demazure_2018} and explored further in the present paper.

\subsection{Schubert and key polynomials}

Given any permutation $w \in S_n$, the \emph{Schubert polynomial} $\mathfrak{S}_w \in \Z[x_1, \dots, x_n]$ can be calculated recursively using a sequence of divided difference operators, by the original definition of Lascoux and Sch\"utzenberger \cite{LS-schuberts}, inspired by the work of Demazure \cite{Demazure} and Bernstein--Gel$'$fand--Gel$'$fand \cite{BGG}. 
Based on a conjecture of Stanley, the first combinatorial formula for Schubert polynomials was given by Billey, Jockusch, and Stanley using the language of \emph{rc-graphs} \cite{billey_combinatorial_1993}, with an alternate proof by Fomin and Stanley \cite{fomin_schubert_1994}.  An equivalent combinatorial description for Schubert polynomials was later provided by Fomin and Kirillov \cite{FominKirillov}, rebranded by Knutson and Miller as \emph{reduced pipe dreams} \cite{KnutsonMiller}, following the conventions of Bergeron and Billey \cite{BilleyBergeron}. Besides being attractive ways to visually represent Schubert polynomials, pipe dreams generalize to flag manifolds the role of the semistandard Young tableaux for Grassmannians, while admitting generalizations to other cohomological contexts.

Many combinatorial models for Schubert polynomials also involve a family of operators, which permute the individual monomials. To highlight several examples most closely related to this work, Bergeron and Billey define \emph{chute and ladder moves} on rc-graphs  \cite{BilleyBergeron}, the inspiration for which they attribute to Kohnert's thesis \cite{KohnertThesis}.  Miller provides a mitosis algorithm which lists reduced pipe dreams recursively by induction on the weak order on $S_n$ \cite{Miller}.
Lenart develops operations on biwords which correspond to the coplactic operators on tableaux \cite{lenart_unified_2004}. Morse and Schilling define a family of operators on \emph{reduced factorizations} in \cite{MorseSchilling}, which restricts to an action on Schubert polynomials via the semi-standard key tableaux of Assaf and Schilling \cite{assaf_demazure_2018}. 

All of the operators mentioned above encode useful combinatorics about Schubert polynomials; however, only some of them additionally carry representation-theoretic information. The most natural approach to track the representation theory is often through Kashiwara's \emph{crystals} \cite{KashiwaraCrystal}, which are graphical models for the irreducible representations of a complex semisimple Lie algebra.  Lenart summarizes many results in \cite{lenart_unified_2004} using the language of crystal operators rooted in a \emph{pairing process} on rc-graphs, though the details are carried out via jeu de tacquin on biwords, most naturally associated with the combinatorics of semistandard Young tableaux. Assaf and Schilling prove more explicitly in \cite[Theorem 5.11]{assaf_demazure_2018} that the set of all reduced factorizations for $w \in S_n$ satisfying an additional \emph{cutoff criterion} decomposes as a union of Demazure crystals; we review their result as Theorem \ref{thm:main-rfc}. 

In type $A_{n-1}$, crystals are indexed by a highest weight vector $\lambda$, which is a partition with $n$ parts.  The character of the crystal $B(\lambda)$ is the Schur polynomial $s_\lambda(x_1, \dots, x_n)$. \emph{Demazure crystals} are subsets of $B(\lambda)$ truncated by a permutation $\pi \in S_n$ which restricts the set of permitted operators.  The character of the Demazure crystal $B_\pi(\lambda)$ is the \emph{key polynomial} $\kappa_a(x_1, \dots, x_n)$ indexed by the composition $a = \pi(\lambda)$. 
Explicit combinatorial formulas for key polynomials, typically as sums over certain tableaux, date back to Lascoux and Sch\"utzenberger \cite{LS-keys}.

The decomposition of a combinatorial model for Schubert polynomials into a union of Demazure crystals thus also yields a description of how $\mathfrak{S}_w$ is expressed as a sum of key polynomials $\kappa_a$, as in \cite[Corollary 5.12]{assaf_demazure_2018}. Tableaux versions of such formulas include the original of Lascoux and Sch\"utzenberger \cite{LS-NCschuberts}, a related result of Reiner and Shimozono \cite{reiner_key_1995} on factorized row-frank words, and so on. The main goal of this paper is to provide such a decomposition for Schubert polynomials as sums of key polynomials, expressed in terms of reduced pipe dreams.

\begin{figure}[ht]
    \centering
    \includegraphics{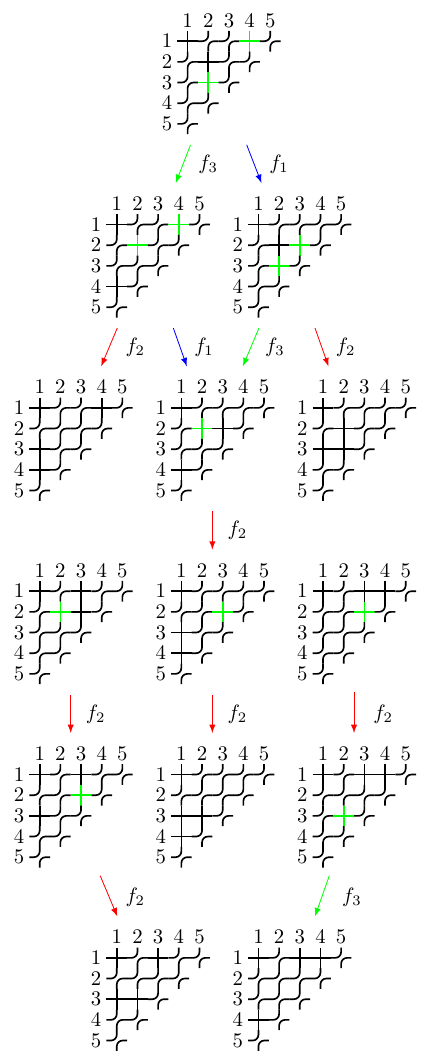}
    \caption{Demazure crystal structure on reduced pipe dreams for $w = [21543]$.}
    \label{fig:main-ex}
\end{figure}

\subsection{Main results}

Inspired by the chute moves of \cite{BilleyBergeron} on rc-graphs,  we develop a crystal structure on the monomials of a Schubert polynomial, giving a method for decomposing Schubert polynomials as sums of key polynomials, complementing the closely related works \cite{assaf_demazure_2018, lenart_unified_2004}.
We summarize this result as Theorem \ref{thm:main-key} below, and we direct the reader to Section \ref{sec:schuberts} for precise definitions of all relevant terminology.

\begin{theorem}\label{thm:main-key} 
Given any $w \in S_n$, the Schubert polynomial may be expressed as 
\begin{equation*}
    \mathfrak{S}_w(x_1, \dots, x_n) = \sum\limits_{\substack{D \in RP(w) \\ e_i(D) = 0,\; \forall 1 \leq i < n}} \kappa_{a_D}(x_1, \dots, x_n),
\end{equation*}
where the composition $a_D$ is uniquely determined by the highest weight pipe dream $D$; see Theorem \ref{thm:AlgCor}.
\end{theorem}

Our \emph{crystal chute moves} are either raising or lowering operators, denoted $e_i$ and $f_i$, respectively.  If the raising operator $e_i(D)$ applied to a reduced pipe dream $D$ for the given permutation $w \in S_n$ equals zero for all $1 \leq i <n$, then we say $D \in RP(w)$ is a \emph{highest weight pipe dream}; see Lemma \ref{lem:inversechutes} for a characterization. The highest weight pipe dreams naturally index the key polynomials in this decomposition, as they are in bijection with a pair given by a partition $\lambda_D$ and permutation $\pi_D \in S_n$ such that $a_D = \pi_D(\lambda_D)$ for a unique composition $a_D$. Figure \ref{fig:main-ex} shows how $\mathfrak{S}_{[21543]}$ decomposes as the sum of three key polynomials, indexed by the three pipe dreams with no incoming lowering edges, having weights $\lambda_D \in \{(2,1,1,0), (2,2,0,0), (3,1,0,0)\}$ recording the number of crosses in each row, with respective truncating permutations $\pi_D \in \{ s_2s_1s_3,s_2,s_3s_2\}$ read from the edges.

The proof of Theorem \ref{thm:main-key} proceeds by showing that the set of reduced pipe dreams for $w$ admits a Demazure crystal structure, determined by applying all possible crystal chute moves as defined in Section \ref{sec:crystalchutes};  see Theorem \ref{thm:main}, from which Theorem \ref{thm:main-key} follows as an immediate corollary. To obtain this crystal structure, we define an equivariant bijection from reduced pipe dreams to reduced factorizations with cutoff, which intertwines with the crystal chute moves.  Theorem \ref{thm:main} then becomes the pipe dream analog of \cite[Theorem 5.11]{assaf_demazure_2018}; equivalently, Theorem \ref{thm:main-key} is the pipe dream analog of \cite[Corollary 5.12]{assaf_demazure_2018}. Although we pass implicitly through the reduced-word compatible sequences of \cite{billey_combinatorial_1993} in Section \ref{sec:rcgraphs}, we highlight Proposition \ref{prop:wtBijPD} and Corollary \ref{cor:PDtoRFCwhole}, which provide formulas directly between reduced pipe dreams and reduced factorizations with cutoff.

In either the language of reduced pipe dreams or reduced factorizations with cutoff, formulas for Schubert polynomials are expressed as a sum over a single combinatorial object, giving rise to the most natural indexing sets for capturing the Demazure crystal structure. In addition, while many related constructions on Schubert polynomials involve a combination of raising and lowering operators, the Demazure crystal structure in Theorem \ref{thm:main} can be obtained by applying exclusively lowering crystal chute moves to a highest weight pipe dream.  Since it is not always possible to reach every vertex in the Demazure crystal by applying raising operators starting with a single node, related results such as \cite[Theorem 3.7]{BilleyBergeron} or \cite[Algorithm 2.6]{lenart_unified_2004} are purely combinatorial in nature, and do not reflect the representation theory captured by Theorem \ref{thm:main-key}.

\subsection{Future directions}

Our primary motivation for developing this particular crystal structure on Schubert polynomials, expressed specifically via the combinatorics of reduced pipe dreams, is to lay the foundation for further connections to both representation theory and geometry.  For example, while many of the chute moves of \cite{BilleyBergeron} are not crystal operators, it would be interesting to understand what geometric meaning non-crystal chute moves might contain, as they stitch together the different components of the Demazure crystal associated to $\mathfrak{S}_w$.

Beyond the original context of the manifold of complete flags in $\C^n$, \cite{KnutsonMillerShimozono} presents a nonnegative formula for the quiver polynomials of \cite{BuchFulton}, which in turn relate to universal, double, and quantum Schubert polynomials.  Experimentation with the bumpless pipe dreams of \cite{LamLeeShimozono} may lead to connections to the back stable Schubert calculus of the infinite flag variety. The origin of our primary inspiration \cite{assaf_demazure_2018} is rooted in the study of Stanley symmetric functions \cite{StanleySymm}, which are stable limits of Schubert polynomials.  In turn, the work of \cite{assaf_demazure_2018} relies on the crystal structure on the set of affine factorizations from \cite{MorseSchilling}, which yields certain Schubert structure constants in the homology of the affine Grassmannian. As such, the combinatorial model of reduced pipe dreams lends itself especially well to generalizations to other cohomological contexts.

\subsection{Organization of the paper}

Section \ref{sec:schuberts} opens with a self-contained introduction to reduced pipe dreams as generating functions for Schubert polynomials. We proceed to develop the theory of crystal chute moves, culminating with the statement of Theorem \ref{thm:main}. We briefly review reduced-word compatible sequences and rc-graphs in Section \ref{sec:rcgraphs}.
The version of Theorem \ref{thm:main} from \cite{assaf_demazure_2018} in terms of reduced factorizations with cutoff is presented in Section \ref{sec:rfcs}, together with a detailed summary of the corresponding crystal operators and the required bijection from reduced-word compatible sequences. The proofs of Theorems \ref{thm:main-key} and \ref{thm:main} then follow in Section \ref{sec:intertwine} by showing that the weight-preserving bijection from reduced pipe dreams to reduced factorizations with cutoff in Proposition \ref{prop:wtBijPD} is equivariant with respect to the crystal chute moves. We present an algorithm in Section \ref{sec:permutation} for obtaining a new diagram from a highest weight pipe dream, which encodes the truncating permutation for the corresponding Demazure crystal; see Theorem \ref{thm:AlgCor}.

\subsection*{Acknowledgements}

The authors wish to thank Anne Schilling and Sarah Mason for helpful correspondence, and Alex Woo for a useful reference. This work has also benefited indirectly from numerous discussions with Jo\~ao Pedro Carvalho, as well as directly by providing his tikz code for chute moves. Many experiments which led to these results were conducted in Sage, and we thank the Sage combinatorics developers for implementing related procedures \cite{SageMath, Sage-combinat}.


\section{A Crystal Structure on Pipe Dreams}\label{sec:schuberts}

In Section \ref{sec:pipedreams}, we review the combinatorics of Schubert polynomials in the language of reduced pipe dreams. We then define crystal chute moves by restricting the chute moves of \cite{BilleyBergeron} on rc-graphs via a pairing process. All necessary properties to show that crystal chute moves define actual crystal operators are proved in Section \ref{sec:crystalchutes}. Finally, we briefly review the construction of Demazure crystals in Section \ref{sec:mainthm}, in order to formally state our main representation-theoretic result as Theorem \ref{thm:main}.

\subsection{Schubert polynomials and pipe dreams}\label{sec:pipedreams}

Lascoux and Sch\"utzenberger originally defined Schubert polynomials recursively \cite{LS-schuberts}. Though non-obvious from the recursive definition, the monomials of Schubert polynomials have non-negative coefficients \cite{billey_combinatorial_1993, fomin_schubert_1994}. 
This paper adopts the viewpoint of Schubert polynomials as generating functions over the rc-graphs of \cite{billey_combinatorial_1993}, equivalently the planar histories or pseudo-line arrangements of \cite{FominKirillov}, subsequently popularized as reduced pipe dreams by \cite{KnutsonMiller}. 

We review the combinatorics of reduced pipe dreams, which index the monomials of Schubert polynomials; see Theorem \ref{thm:PDSchub}. Fix an $n \in \N$ and consider the $n\times n$ grid, indexed such that the box in row $i$ from the top and column $j$ from the left is labeled by $(i,j)$, as for matrix entries. A \emph{pipe dream} is a diagram $D$ obtained by covering each box on the grid with one of two square tiles: a cross \includegraphics[height=\fontcharht\font`\B]{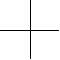}  or an elbow \includegraphics[height=\fontcharht\font`\B]{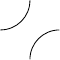}. Further, crosses are only permitted in boxes $(i,j)$ such that $i+j \leq n$, so we will typically only draw the portion of $D$ which lies on or above the main anti-diagonal. 
By connecting the crosses and elbows on each tile in the unique possible way, as shown in Figure \ref{fig:PDex}, we can view the resulting diagram as a network of \emph{pipes} moving north and east, with water flowing in from the left of the grid and out at the top. 

\begin{figure}[h!]
    \centering
    \includegraphics{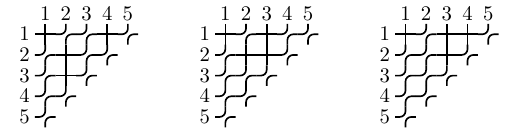}
    \caption{Several reduced pipe dreams for $w = [21543] \in S_5$.}
    \label{fig:PDex}
\end{figure}

The water in each pipe enters and exits from a unique pair of row and column indices, so that each pipe dream corresponds to a permutation on the set $[n] = \{1, \dots, n\}$ as follows. The one-line notation for a permutation $w \in S_n$ records the action of $w$ on $[n]$ in the form $w = [w_1 \cdots w_n]$, where we write $w_i = w(i)$ for brevity. A diagram $D$ is a pipe dream \emph{for the permutation} $w = [w_1 \cdots w_n]$ if the pipe entering row $i$ exits from column $w_i$ for all $i \in [n]$. Each of the three diagrams in Figure \ref{fig:PDex} is a pipe dream for the same permutation $w = [21543] \in S_5$.

A pipe dream is \emph{reduced} if each pair of pipes crosses at most once, as in Figure \ref{fig:PDex}. Denote by $RP(w)$ the set of all reduced pipe dreams for a given permutation $w$. We denote by $D_+$ the set of all boxes of $D$ which are covered by a cross; note that $D_+$ uniquely determines $D$.  Provided that the pipe dream is reduced, \cite[Lemma 1.4.5]{KnutsonMiller} says that the number of crosses in $D \in RP(w)$ equals the length of the permutation, or the number of its inversions, given by $|D_+| = \ell(w) = \#\{ i < j \mid w_i > w_j\}$.

The \emph{weight of a pipe dream} $D \in RP(w)$, denoted by $\wt(D)$, is the weak composition of $\ell(w)$ whose $i^{\text{th}}$ coordinate equals the number of crosses in row $i$ of $D$.  For example, the three weight vectors corresponding to the pipe dreams from Figure \ref{fig:PDex} are $(2,1,1,0), (2,2,0,0),$ and $(3,1,0,0)$ recorded from left to right, all of which happen to be partitions in this example. In general, given any weak composition $a$ with $n$ parts, there exists a unique \emph{shortest} permutation $\pi \in S_n$ such that $\pi^{-1}(a)$ is a partition, meaning that $\ell(\pi)$ is minimal among all such permutations.

Schubert polynomials can be viewed as generating functions over pipe dreams, as illustrated by the following result, originally proved by Billey, Jockusch and Stanley \cite{billey_combinatorial_1993} as Theorem \ref{thm:RCSchub}, later reproved by Fomin and Stanley \cite{fomin_schubert_1994}, and recorded here in the language of pipe dreams.

\begin{theorem}[Corollary 2.1.3 \cite{KnutsonMiller}]\label{thm:PDSchub}
Let $w \in S_n$. Then
\begin{equation}\label{eq:SchubPD}
    \mathfrak{S}_w(x_1, \dots, x_n) = \sum\limits_{D \in RP(w)} \mathbf{x}^{\wt(D)}.
\end{equation}
\end{theorem}

\noindent We use $\mathbf{x}$ to denote a monomial in the variables $x_1, \dots, x_n$. Given any vector $\mathbf{v} = (v_1, \dots, v_n) \in \Z_{\geq 0}^n$, the notation $\mathbf{x}^{\mathbf{v}} = x_1^{v_1}\cdots x_n^{v_n}$ is used throughout.

Since the pipe dreams in Equation \eqref{eq:SchubPD} are reduced, the Schubert polynomial $\mathfrak{S}_w$ is homogeneous of degree $\ell(w)$. The reader who is new to Schubert polynomials can simply take Equation \eqref{eq:SchubPD} to be the definition of $\mathfrak{S}_w$. We explain how to use Theorem \ref{thm:PDSchub} to produce a Schubert polynomial in the following example.

\begin{example}\label{ex:SchubPD}
The reduced pipe dreams shown in Figure \ref{fig:PDex} represent 3 of the 14 elements of $RP(w)$ for the permutation $w = [21543] \in S_5$. In the lefthand figure $D$, we see 4 crosses occurring in boxes $(1,1), (1,4), (2,2)$ and $(3,2)$.  Recording their row indexes as $\wt(D) = (2,1,1,0)$, we obtain $\mathbf{x}^{\wt(D)} = x_1^2x_2x_3$ as one monomial of $\mathfrak{S}_w$. The middle and righthand pipe dreams in Figure \ref{fig:PDex} contribute $x_1^2x_2^2$ and $x_1^3x_2$, respectively, to the Schubert polynomial $\mathfrak{S}_w$; see Example \ref{ex:RCSchub} for the complete formula for this particular $\mathfrak{S}_w$.
\end{example}

\subsection{Crystal chute moves}\label{sec:crystalchutes}

In this section, we describe a family of operators on the set $RP(w)$ of reduced pipe dreams for a given permutation, which we show in our main theorem  produces a Demazure crystal structure on the monomials of $\mathfrak{S}_w$.

\begin{definition}\label{def:pairing}
Given a reduced pipe dream \(D\) for a permutation in $S_n$, we fix a row index $i \in [n]$. Denote the rightmost cross in row \(i\) by $c$.  (Since crosses only occur in boxes $(i,j)$ such that $i + j \leq n$, then $D$ has no crosses in row $n$.) We define a \emph{pairing process on row} $1 \leq i <n$ of $D$ as follows:
	 \begin{enumerate}
        \item Look for an unpaired cross $c_+$ in row \(i+1\) whose column index is greater than or equal to that of $c$, so that $c_+$ lies below and weakly to the right of $c$ in the diagram $D$. If there are multiple such $c_+$, choose the leftmost $c_+$.
   \begin{enumerate}
       \item If such a $c_+$ exists, we say that $c$ and $c_+$ are \emph{paired}. 
       \item If no such $c_+$ exists, we say that $c$ is \emph{unpaired}.
   \end{enumerate}
		 \item Denote by $c'$ the cross in row $i$ which is both closest to $c$ and lies to the left of $c$.
   \begin{enumerate}
       \item If such a $c'$ exists, we reset \(c := c'\) and start again from step (1). 
       \item If no such $c'$ exists, the pairing process on row $i$ is complete.
   \end{enumerate}
	\end{enumerate}
\end{definition}

We illustrate the pairing process on the righthand pipe dream from Figure \ref{fig:PDex}.

\begin{example}\label{ex:pairing}
Fix $i=1$ and identify $c = (1,4)$ as the rightmost cross in row 1. Since there are no crosses in row 2 which lie weakly to the right of $c$, then $c_+$ does not exist and $c$ is unpaired in step (1b).  We indicate this by coloring $c=(1,4)$ in red in the figure below.

 \begin{figure}[h!]
    \centering
    \includegraphics{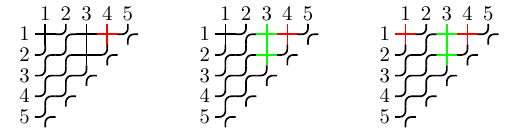}
    \caption{The pairing process applied to row 1 of a reduced pipe dream.}
    \label{fig:pairing}
\end{figure}

In step (2), we identify $c' = (1,3)$ as the cross in row 1 closest to and left of the original $c=(1,4)$. We thus return to step (1) applied to $c=(1,3)$ instead. In this case, we identify $c_+ = (2,3)$ as a cross in row 2 which is weakly right of $c = (1,3)$, and so these crosses get paired in step (1a). We indicate this pairing by coloring the crosses green.

Turning now to step (2), we identify $c' = (1,1)$ as the only remaining cross in row 1, and we return to step (1) with $c=(1,1)$. However, all crosses in row 2 which are weakly to the right of $c=(1,1)$ have already been paired, and so no unpaired $c_+$ exists. We color $c=(1,1)$ red to indicate that it is unpaired in step (1b).  The pairing process is now complete, having analyzed all crosses in the designated row 1.
\end{example}

The outcome of an individual cross in the pairing process is highly dependent on the initial choice of the row. For example, although the cross $(2,3)$ is paired in Example \ref{ex:pairing}, it is unpaired when we run the pairing process on row 2, as there are no crosses weakly right of it in row 3.

After running the pairing process on row $i$ of $D \in RP(w)$, we define an operator $f_i$ on $D$ which produces another element of $RP(w)$ whenever it is nonzero; see Lemma \ref{prop:f_iwelldef} below. This family of operators forms a special case of the chute moves on rc-graphs originally defined in \cite[Section 3]{BilleyBergeron}, which explains our choice of nomenclature. Billey and Bergeron further credit a conjecture in Kohnert's thesis \cite{KohnertThesis} as their inspiration for chute moves.

\begin{definition}\label{def:crystalchutes} 
    Let $D \in RP(w)$ for $w \in S_n$. Fix an $1 \leq i < n$ and run the pairing process on row $i$ of $D$.
    If all crosses in row $i$ are paired, then set $f_i(D) = 0$.  Otherwise, denote by $(i,j)\in D_+$ the leftmost unpaired cross in row $i$. 
    
    If $(i,k) \in D_+$ for all $1 \leq k \leq j$, then set $f_i(D) = 0$. Otherwise, define $m \in \N$ such that:
    \begin{enumerate}
        \item[(a)] $(i,j-m), (i+1, j-m) \notin D_+$ and
        \item[(b)] $(i,j-k), (i+1,j-k) \in D_+$ for all $1 \leq k <m$.
    \end{enumerate}
    Define a new diagram $f_i(D)$ by
    \begin{equation*}
        f_i(D)_+ = D_+ \backslash \left\{ (i,j) \right\} \cup \{(i+1, j-m)\}.
    \end{equation*}
    The family of operators $f_i$ for $1 \leq i <n$ are called \emph{(lowering) crystal chute moves}.
\end{definition}

    In words, the crystal chute move $f_i$ exchanges the leftmost unpaired cross at $(i,j)$ and the elbow at $(i+1, j-m)$, where $m$ is chosen such that the rectangle strictly between this pair of tiles is filled by crosses.  See Figure \ref{fig:chute} below for an illustration of the local effect of a crystal chute move.
    
\begin{figure}[h!]
    \centering
    \includegraphics[width=4in]{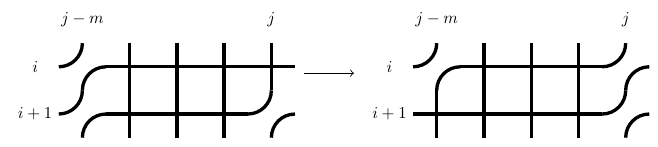}
    \caption{A (crystal) chute move swaps the unpaired cross on the upper right with the elbow on the lower left, preserving the rectangle of crosses in between. 
    }
    \label{fig:chute}
\end{figure}

\begin{remark}
In \cite{BilleyBergeron}, chute moves are defined on all crosses such that the integer $m \in \N$ is defined, without respect to any pairing process.  The advantage of their more generous application of chute moves is that all monomials of the Schubert polynomial $\mathfrak{S}_w$ can be obtained by applying all possible chute moves to a single element $D_{\operatorname{top}} \in RP(w)$ as in \cite[Theorem 3.7]{BilleyBergeron}.  In Theorem \ref{thm:main} below, however, we show that utilizing only the crystal chute moves produces a Demazure crystal structure on the monomials of $\mathfrak{S}_w$. 
\end{remark}

\begin{example}\label{ex:fpipes}
We illustrate how to apply the crystal chute moves on several different pipe dreams.

\begin{enumerate}    
    \item First recall the pipe dream $D$ for $w = [21543]$ from Example \ref{ex:pairing}, for which the pairing process on row 1 is depicted in Figure \ref{fig:pairing}. Two crosses in row 1 remain unpaired, the leftmost of which is $(1,1) \in D_+$. However, since $(1,k) \in D_+$ for all $1 \leq k \leq 1$, then $f_1(D) = 0$ by Definition \ref{def:crystalchutes}.  
    
    By contrast, there is a \emph{nonzero} chute move on row 1 of $D$ in the sense of \cite{BilleyBergeron}, which would instead move the rightmost unpaired cross at $(1,4) \in D_+$ over the rectangle of green crosses in Figure \ref{fig:pairing} and into position $(2,2)$. This chute move is not a \emph{crystal} chute move, however, as it does not respect our pairing process.

    \item To provide another example more closely resembling the general setup depicted in Figure \ref{fig:chute}, consider the sequence shown in Figure \ref{fig:crystalchute-ex}, in which we instead begin with the lefthand pipe dream $D$ for $w = [21543]$ from Figure \ref{fig:PDex}.  If we run the pairing process on row 1, the leftmost unpaired cross is $(1,4) \in D_+$. Properties (a) and (b) hold for $m=1$, and the corresponding rectangle of crosses between $(1,4) \in D_+$ and the elbow at $(2,3)$ is empty in this case. To apply $f_1$, the red cross in $(1,4)$ moves to the blue elbow in position $(2,3)$, resulting in the middle diagram in Figure \ref{fig:crystalchute-ex}.  

  \begin{figure}[ht]
    \centering
    \includegraphics{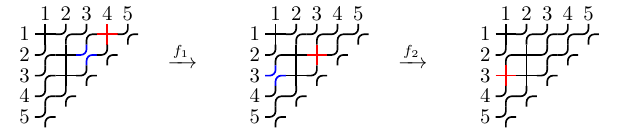}
    \caption{Applying a sequence of crystal chute moves to a reduced pipe dream.}
    \label{fig:crystalchute-ex}
    \end{figure}

    Running the pairing process now instead on row 2 of the middle pipe dream, $(2,3)$ is the leftmost unpaired cross, and $m=2$, corresponding to the tile of paired crosses in rows 2 and 3 which are preserved under applying $f_2$.  Here instead, the red cross at $(2,3)$ jumps over this rectangle of crosses to the blue elbow in position $(3,1)$, resulting in the third diagram in the figure above.
    \end{enumerate}
\end{example}

The following proposition justifies our choice of terminology for the crystal chute moves.

\begin{proposition}\label{prop:f_iwelldef}
    The lowering crystal chute move $f_i: RP(w) \to RP(w) \cup \{0\}$ is a well-defined map for all $1 \leq i < n$. Moreover, for any $D \in RP(w)$ such that $f_i(D) \neq 0$, we have $\wt(f_i(D)) = \wt(D) - \alpha_i$, where $\alpha_i = e_i-e_{i+1}$.
\end{proposition}

\begin{proof}
    In case there are either no unpaired crosses in row $i$ after running the pairing process on row $i$ of $D$ or all tiles to the left of the leftmost unpaired $(i,j) \in D_+$ are also crosses, we have $f_i(D) = 0$.  Henceforth, we thus assume that there exists an unpaired cross in row $i$ and that there exists $1 \leq k < j$ such that $(i,j-k) \notin D_+$.
    
    We first claim that the integer $m \in \N$ is well-defined.  If the tiles to the left of $(i,j) \in D_+$ in rows $i$ and $i+1$ are not both elbows, we claim that they are both crosses.  If $(i,j-1) \in D_+$ but $(i+1,j-1) \notin D_+$, then the cross at $(i,j-1)$ is unpaired and lies farther left than $(i,j)$, a contradiction. Now suppose $(i,j-1) \notin D_+$ but $(i+1,j-1) \in D_+$. Note that $(i,j)$ being unpaired forces $(i+1,j) \notin D_+$, and so we have elbows on the diagonal and crosses on the anti-diagonal of the square in rows $i$ and $i+1$ and columns $j-1$ and $j$. Connecting pipes as in Figure \ref{fig:nonreduced},
     \begin{figure}[h]
    \centering
    \includegraphics{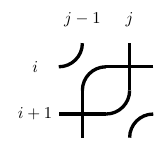}
    \caption{Rectangles of crosses not bordered by elbows are non-reduced.}
    \label{fig:nonreduced}
    \end{figure} 
    we see that the same pair if pipes crosses at both $(i+1,j-1)$ and $(i,j)$, contradicting $D \in RP(w)$. Therefore, the tiles immediately left of $(i,j)$ in rows $i$ and $i+1$ are either both elbows or both crosses.

    If $(i,j-1), (i+1, j-1) \notin D_+$, then property (a) holds and (b) is vacuously true for $m=1$.  If $(i,j-1), (i+1, j-1) \in D_+$, the same argument as above proves that the tiles to the left of these two are either both elbows or both crosses. Moving left through rows $i$ and $i+1$ and repeating this argument, the first pair of elbows we obtain determines the integer $m$ such that both $(i,j-m), (i+1,j-m) \notin D_+$ and $(i,j-k), (i+1,j-k) \in D_+$ for all $1 \leq k <m$. The fact that such an $m$ exists is guaranteed by the hypothesis that $(i,j-k) \notin D_+$ for some $1 \leq k < j$.

    Next we show that $f_i(D) \in RP(w)$. The operator $f_i$ preserves the permutation associated to the pipe dream $D \in RP(w)$, as can be seen by noting in Figure \ref{fig:chute} that the pipes all enter (resp.~exit) from the same rows (resp.~columns) both before and after applying the operator; see also \cite[Lemma 3.5]{BilleyBergeron}. In addition, by definition $|D_+| = |f_i(D)_+|$ so that the pipe dream $f_i(D)$ for $w$ is reduced via \cite[Lemma 1.4.5]{KnutsonMiller}.

    Finally, recall that $\wt(D)$ is the weak composition whose $i^{\text{th}}$ coordinate equals the number of crosses in row $i$ of $D$. Applying $f_i(D)$ moves one cross out of row $i$ and into row $i+1$, so that $\wt(f_i(D)) = \wt(D) - e_i + e_{i+1} = \wt(D) - \alpha_i$.
\end{proof}

We now provide a criterion for a pipe dream to be the result of applying a crystal chute move, giving rise in Definition \ref{def:inversechutes} to an inverse to the crystal chute moves; compare \cite[Lemma 3.6]{BilleyBergeron}.

\begin{lemma}\label{lem:inversechutes}
    Let $D \in RP(w)$ for some $w \in S_n$. Then $D$ is the result of a crystal chute move $f_i$ for some $1 \leq i < n$ if and only if there exists an unpaired cross in row $i+1$, after running the pairing process on row $i$ of $D$.
\end{lemma}

\begin{proof}
    Let $w \in S_n$, and run the pairing process on row $i$ of $D \in RP(w)$ for $1 \leq i < n$.
    
    First suppose that there exists an unpaired cross in row $i+1$, say $(i+1, \ell) \in D_+$. From there, scan right across row $i+1$ to find the first elbow tile, say $(i+1,n) \notin D_+$. Note that a minimal such $n \in \N$ exists since all tiles on the main anti-diagonal are elbows. By construction, we then have $(i+1,k) \in D_+$ for all $\ell \leq k < n$. Note that $(i, k) \in D_+$ for all $\ell \leq k < n$ as well, since otherwise one of $(i+1,k) \in D_+$ would be an unpaired cross in row $i+1$ farther to the right of $(i+1,\ell)$. If $(i,n)$ were a cross, then the same pair of pipes crosses at both $(i+1,\ell)$ and $(i,n)$, contradicting $D \in RP(w)$; see Figure \ref{fig:nonreduced}. Therefore, $(i,n) \notin D_+$ is an elbow tile. Finally, $(i,\ell) \notin D_+$ since otherwise the cross at $(i+1,\ell)$ would be paired. In rows $i$ and $i+1$, we thus have a (possibly empty) rectangle of crosses between columns $\ell$ and $n$, whose corners are three elbow tiles and one cross in position $(i+1,\ell)$.

    We claim that $D = f_i(D')$, where $D_+' = D_+ \backslash \left\{ (i+1,\ell) \right\} \cup \{(i, n)\}$. By construction, all pipes still enter and exit from the same rows and columns in both $D$ and $D'$. As the number of crosses in $D'$ is the same as in $D$, then $D' \in RP(w)$. Now run the pairing process on row $i$ of $D'$. Because $(i+1,\ell)$ was the rightmost unpaired cross in row $i+1$ of $D$, then the cross $(i,n) \in D'_+$ is the leftmost unpaired cross in row $i$ of $D'$. Further note that $(i,\ell) \notin D'_+$ where $\ell < n$, and so $f_i(D') \neq 0$. By the argument in the previous paragraph, we see that $m=n-\ell$ satisfies properties (a) and (b) of Definition \ref{def:crystalchutes} when $j=n$.  Therefore, $f_i(D') = D'_+ \backslash \{ (i,n)\} \cup \{ (i+1, \ell)\} = D$, and so $D$ is indeed the result of a crystal chute move.

    Conversely, suppose that $D = f_i(D') \neq 0$ for some $D' \in RP(w)$. After running the pairing process on row $i$ of $D'$, let $(i,j)\in D'_+$ denote the leftmost unpaired cross in row $i$ of $D'$, so that $D_+ = D'_+ \backslash \{(i,j)\} \cup \{ (i+1, j-m)\}$ for $m \in \N$ satisfying (a) and (b) of Definition \ref{def:crystalchutes}.  Note that the cross $(i+1, j-m) \in D_+$ is the rightmost unpaired cross in row $i+1$ after running the pairing process on row $i$ of $D$, since $(i,j)$ was the leftmost unpaired cross in row $i$ of $D'$, and the tiles in columns strictly between form a rectangle of crosses by property (b) of $f_i(D')$. Therefore, $D$ has an unpaired cross in row $i+1$, completing the proof. 
\end{proof}

We now define a second family of operators $e_i$ to be precisely the inverse crystal chute moves characterized by Lemma \ref{lem:inversechutes}.

\begin{definition}\label{def:inversechutes}
Let $D \in RP(w)$ for $w \in S_n$. Fix an $1 \leq i < n$ and run the pairing process on row $i$ of $D$. If all crosses in row $i+1$ are paired, then set $e_i(D)=0$. Otherwise, denote by $(i+1,\ell) \in D_+$ the rightmost unpaired cross in row $i+1$.

Let $q>\ell$ be minimal such that $(i+1,q) \notin D_+$. Define a new diagram $e_i(D)$ by
   \begin{equation*}
             e_i(D)_+ = D_+ \backslash \left\{ (i+1,\ell) \right\} \cup \{(i, q)\}.
   \end{equation*}
   The family of operators $e_i$ for $1 \leq i < n$ are called \emph{(raising) crystal chute moves}.
\end{definition}

The well-definedness of the operators $e_i$ was the primary purpose of Lemma \ref{lem:inversechutes}, but we record the following explicit statement for future use.

\begin{proposition}\label{prop:e_iwelldef}
    The raising crystal chute move $e_i: RP(w) \to RP(w) \cup \{0\}$ is a well-defined map for all $1 \leq i < n$, satisfying $\wt(e_i(D)) = \wt(D) + \alpha_i$ for any $D \in RP(w)$. Moreover, the raising and lowering crystal chute moves are mutually inverse, meaning $e_i(D) = D'$ if and only if $f_i(D') = D$ for $D, D' \in RP(w)$.
\end{proposition}

\begin{proof}
    In case there are no unpaired crosses in row $i+1$ after running the pairing process on row $i$ of $D$, then $e_i(D) = 0$. 
    Otherwise, Lemma \ref{lem:inversechutes} says that $D$ is the image of the crystal chute move $f_i$. Moreover, the proof of Lemma \ref{lem:inversechutes} shows that the preimage $f_i^{-1}(D)$ is precisely the pipe dream $e_i(D)_+ = D_+ \backslash \left\{ (i+1,\ell) \right\} \cup \{(i, q)\},$ and so $e_i(D) \in RP(w) \cup \{0\}$ is well-defined.  
    Since $\wt(f_i(D)) = \wt(D)-\alpha_i$ by Proposition \ref{prop:f_iwelldef} and $e_i(D) = f_i^{-1}(D)$ by Lemma \ref{lem:inversechutes}, then $\wt(e_i(D)) = \wt(D)+\alpha_i$. The inverse relationship of $e_i$ and $f_i$ follows directly from Lemma \ref{lem:inversechutes} and Definition \ref{def:inversechutes}.
\end{proof}

\begin{remark}\label{rem:lenart}
Our raising and lowering crystal chute moves coincide with the operators on rc-graphs denoted by $\tilde{e}_r, \tilde{f}_r$ in \cite{lenart_unified_2004}, though the authors discovered them independently via the explicit crystal structure on reduced factorizations with cutoff in \cite{assaf_demazure_2018}. Compare our treatment to \cite[Section 3]{lenart_unified_2004}, in which the properties of $\tilde{e}_r, \tilde{f}_r$ are developed using the corresponding operators on biwords via the plactification map.
\end{remark}

The pipe dreams on which $e_i = 0$ for all $1 \leq i < n$ play a distinguished role in the statement of Theorem \ref{thm:main} below, so we highlight them here.

\begin{definition}
     If $e_i(D)=0$ for all $1 \leq i <n$, then we refer to $D$ as a \emph{highest weight pipe dream}. If $f_i(D)=0$ for all $1 \leq i <n$, then we refer to $D$ as a \emph{lowest weight pipe dream}. 
\end{definition}

We now record several lemmas concerning highest weight pipe dreams, for future use. 

\begin{lemma}\label{lem:HWPD}
    If $D \in RP(w)$ is a highest weight pipe dream, then $\wt(D)$ is a partition.
\end{lemma}

\begin{proof}
    Recall from Definition \ref{def:inversechutes} that $e_i(D) = 0$ if and only if all crosses in row $i+1$ are paired after running the pairing process on row $i$ of $D$.  If row $i+1$ contains strictly more crosses than row $i$, then there will be at least one unpaired cross in row $i+1$.  Therefore, the number of crosses must weakly decrease, enumerating rows from top to bottom; equivalently, $\wt(D)$ is a partition.
\end{proof}

\begin{lemma}\label{lem:HWPDincseq}
    Let $D \in RP(w)$ be a highest weight pipe dream. Let $c = (i,j_\ell^i) \in D_+$ denote the position of the $\ell^{\text{th}}$ cross from the left in row $i$ of $D$. Then for any fixed $1 \leq \ell \leq \wt(D)_1$, the column sequence $(j_\ell^i)$ is weakly increasing.   
\end{lemma}

\begin{proof}
    Let $D$ be a highest weight pipe dream. We must show that the column indices of the $\ell^{\text{th}}$ entries in each row of $D$ are weakly increasing. Suppose by contradiction that the $\ell^{\text{th}}$ crosses in rows $i$ and $i+1$ occur in columns $j^i_\ell >j^{i+1}_\ell$.  Since crosses in row $i$ are paired with crosses in row $i+1$ which lie weakly to the right, then the cross in column $j^{i+1}_\ell$ must be unpaired when running the pairing process on row $i$, contradicting the fact that $D$ is highest weight via Lemma \ref{lem:inversechutes}. Therefore, $j^i_\ell \leq j^{i+1}_\ell$ for any $1 \leq \ell \leq \wt(D)_1$ and all $1 \leq i < n$.
\end{proof}

\subsection{Demazure crystals and the main theorem}\label{sec:mainthm}

Given a partition $\lambda$ with $n$ parts, the type $A_{n-1}$ crystal of highest weight $\lambda$ is denoted by $B(\lambda)$, and the character of the crystal $B(\lambda)$ is the Schur polynomial $s_\lambda(x_1, \dots, x_n)$. We refer the reader to \cite{bump_crystal_2017} for more background on crystals. 

Demazure crystals are subsets of $B(\lambda)$ truncated by a permutation which restricts the set of raising and lowering operators. More precisely, for any subset $X \subseteq B(\lambda)$ and any index $1 \leq i < n$, we define $\mathfrak{D}_i$ in terms of lowering operators as
\begin{equation*}
\mathfrak{D}_i (X) = \{ b \in B(\lambda) \mid b \in f^k_i(X)\; \text{for some}\; k \geq 0\}.
\end{equation*}
Now given any $\pi \in S_n$, write $\pi = s_{i_1} \cdots s_{i_p}$ as a product of simple transpositions $s_i = (i, i+1)$ where the expression for $\pi$ is reduced, meaning that $p= \ell(\pi)$ is minimal. If $u_\lambda$ denotes the highest weight element of $B(\lambda)$, the \emph{Demazure crystal} associated to the pair $(\lambda, \pi)$ is defined by
\begin{equation*}
    B_\pi(\lambda) = \mathfrak{D}_{i_1} \cdots \mathfrak{D}_{i_p}(u_\lambda).
\end{equation*}
The character of the Demazure crystal $B_\pi(\lambda)$ generalizes the Demazure characters of \cite{Demazure-characters}, as conjectured by Littelmann \cite{Littelmann} and proved by Kashiwara \cite{Kashiwara}. Moreover, the character of the Demazure crystal $B_\pi(\lambda)$ is the key polynomial $\kappa_a(x_1, \dots, x_n)$ indexed by the composition $a = \pi(\lambda)$. The Schur polynomials occur as the special case where $\pi = w_0$ is the longest permutation $w_0 = [n\; n-1 \cdots 2\; 1] \in S_n$, so that $\pi$ places no restrictions on the lowering operators defining $B_{w_0}(\lambda) = B(\lambda)$.

Our main theorem says that the set of reduced pipe dreams for a given permutation admits a Demazure crystal structure determined by the crystal chute moves from Section \ref{sec:crystalchutes}.

\begin{theorem}\label{thm:main} 
Given any $w \in S_n$, the operators $e_i$ and $f_i$ for $1 \leq i < n$ define a type $A_{n-1}$ Demazure crystal structure on $RP(w)$. That is, 
\begin{equation*}
    RP(w) \cong \bigcup\limits_{\substack{D \in RP(w) \\ e_i(D) = 0,\; \forall 1 \leq i < n}} B_{\pi_D}(\wt(D)),
\end{equation*}
where the truncating permutation $\pi_D$ is the shortest permutation such that $\operatorname{wt}(\widetilde{D}) = \pi_D(\wt(D))$, for a diagram $\widetilde{D}$ uniquely determined by the highest weight pipe dream $D$; see Theorem~\ref{thm:AlgCor}.
\end{theorem}

\noindent Theorem \ref{thm:main} is the pipe dream analog of \cite[Theorem 5.11]{assaf_demazure_2018}, phrased there in terms of reduced factorizations for $w$ meeting a certain cutoff condition. 

\begin{remark}
   We emphasize that the Demazure crystal structure on the set $RP(w)$ is truly innate. That is, the full crystal structure on $RP(w)$ can be obtained by simply applying all possible crystal chute moves to all elements of $RP(w)$, without any knowledge of either highest weight pipe dreams or the truncating permutation. However, in Section \ref{sec:permutation}, we offer a convenient algorithm for recording the truncating permutation $\pi_D$ directly from a highest weight pipe dream $D$, giving an alternative way to decompose $RP(w)$ into the same union of Demazure crystals.
\end{remark}

We conclude the section by illustrating Theorem \ref{thm:main} in an example.

\begin{example}\label{ex:mainthm}
Figure \ref{fig:main-ex} shows how the set $RP\left([21543]\right)$ decomposes as the disjoint union of three Demazure crystals, indexed by the three highest weight pipe dreams depicted with no incoming $f_i$ edges, having weights $\lambda_D \in \{(2,1,1,0), (2,2,0,0), (3,1,0,0)\}$. The three respective truncating permutations are $\pi_D \in \{ s_2s_1s_3,s_2,s_3s_2\}$, which can be read from the edges in Figure \ref{fig:main-ex} by recording the corresponding simple reflections from right to left. See Theorem \ref{thm:AlgCor} to obtain $\pi_D$ directly from each of the highest weight pipe dreams, as an alternative method.
\end{example}

The remainder of the paper is dedicated to the proof of Theorem \ref{thm:main}, from which Theorem \ref{thm:main-key} follows as an immediate corollary.


\section{Compatible Sequences and RC-Graphs}\label{sec:rcgraphs}

This section reviews the terminology of compatible sequences and rc-graphs, which are equivalent to the reduced pipe dreams explored in Section \ref{sec:schuberts}.
In Theorem \ref{thm:RCSchub}, we provide an alternate formula for Schubert polynomials due to \cite{billey_combinatorial_1993}, constructed from pairs of reduced words and compatible sequences. We then review the construction of rc-graphs from \cite{BilleyBergeron}, equivalently the planar histories of \cite{FominKirillov}, which provides a weight-preserving bijection between pipe dreams and reduced-word compatible sequences.

\subsection{Compatible sequences}

Given $w \in S_n$ and any reduced expression $w = s_{a_1}s_{a_2} \cdots s_{a_p}$, the sequence of indices $a_1 a_2 \cdots a_p \in [n-1]^{\ell(w)}$ is a \emph{reduced word} for the permutation $w$. We denote the set of all reduced words for $w \in S_n$ by $R(w)$. 

We now review the notion of a reduced-word compatible sequence from \cite{billey_combinatorial_1993}. Given any reduced word $\bfa = a_1\cdots a_p \in R(w)$, a tuple of positive integers $\beta = \beta_1 \cdots \beta_p \in \N^p$, is $\bfa$-\emph{compatible} if the following three properties hold:
    \begin{enumerate}
        \item[(i)] $\beta_1 \leq \cdots \leq \beta_p$,
        \item[(ii)] $\beta_j \leq a_j$ for all $j \in [p]$, and
        \item[(iii)] $\beta_j < \beta_{j+1}$ if $a_j < a_{j+1}$.
    \end{enumerate}
    We denote the set of $\bfa$-compatible sequences by $C(\bfa)$.

\begin{example}\label{ex:compatible}
    	All reduced words of \(w = [21543] \in S_5\) and their compatible sequences are provided in the table below. The elements $\bfa \in R(w)$ are listed along the first row, and the $\bfa$-compatible sequences appear below in corresponding columns, whenever they exist.
     \vskip 10pt
	\begin{center}
		\begin{tabular}{cccccccc}
			3431 & 4341 & 3413 & 3143 & 1343 & 4314 & 4134 & 1434 \\ \hline 
			$\emptyset$ & $\emptyset$ & $\emptyset$ & 1122 & 1233 & 1112 & 1123 & 1223 \\
			     &     &      &  1123 &      & 1113 & 1124 & 1224 \\
			     &    &      &   1133 &      & 1114 & 1134 & 1234 \\
			     &    &     &         &     &       &      & 1334
        \end{tabular}
	\end{center}
\end{example}

Let $\bfa \in R(w)$ for some $w \in S_n$.  The \emph{weight of an $\bfa$-compatible sequence} $\beta \in C(\bfa)$, denoted $\wt(\beta)$, is the weak composition of $\ell(w)$ whose $i^{\text{th}}$ coordinate equals the number of $i$'s the sequence $\beta$ contains.  For example, if $w = [21543] \in S_5$ and $\bfa = 1434$, then $\beta = 1224$ is $\bfa$-compatible and $\wt(\beta) = (1,2,0,1)$.

The following formula for Schubert polynomials as generating functions over reduced-word compatible sequences was conjectured by Stanley, originally proved by Billey, Jockush, and Stanley \cite{billey_combinatorial_1993}, and reproved by Fomin and Stanley in an elegant alternate way in \cite{fomin_schubert_1994}. 

\begin{theorem}[Theorem 1.1 \cite{billey_combinatorial_1993}] \label{thm:RCSchub}
Let $w \in S_n$. Then
\begin{equation*}
    \mathfrak{S}_w(x_1, \dots, x_n) = \sum_{\bfa \in R(w)} \sum_{\beta \in C(\bfa)} \mathbf{x}^{\wt(\beta)}.
\end{equation*}
\end{theorem}

We now complete the Schubert polynomial calculation from Example \ref{ex:SchubPD} using Theorem \ref{thm:RCSchub}.

\begin{example}\label{ex:RCSchub}
    Each compatible sequence recorded in Example \ref{ex:compatible} corresponds to a monomial of the Schubert polynomial $\mathfrak{S}_w$ for the permutation $w = [21543] \in S_5$ by Theorem \ref{thm:RCSchub}. For example, the compatible sequences for the reduced word $\bfa = 3143$ are $\beta \in \{1122, 1123, 1133\}$ with weights $\wt(\beta) \in \{(2,2,0,0),(2,1,1,0),(2,0,2,0)\}$, which correspond to the monomials $\mathbf{x}^{\wt(\beta)} \in \{x_1^2x_2^2, x_1^2x_2x_3, x_1^2x_3^2\}$, respectively.
    Continuing in this manner, we obtain 
    \begin{align*}
    \mathfrak{S}_{[21543]} = & \ \ x_1^2x_2^2 + 2x_1^2x_2x_3 + x_1^2x_3^2+ x_1x_2x_3^2+x_1^3x_2+x_1^3x_3 +x_1^3x_4 \\ & +x_1^2x_2x_4+x_1^2x_3x_4+x_1x_2^2x_3+x_1x_2^2x_4+x_1x_2x_3x_4+x_1x_3^2x_4.
    \end{align*}
\end{example}

Theorem \ref{thm:RCSchub} was presented in the language of pipe dreams as Theorem \ref{thm:PDSchub} earlier in the paper, although the formulation in terms of compatible sequences came first chronologically. The precise connection between the two results is made using a graphical depiction of a reduced-word compatible sequence pair due to Bergeron and Billey \cite{BilleyBergeron}, which we explain in the next section.

\subsection{rc-graphs and pipe dreams}

Let $w \in S_n$, and fix $\bfa \in R(w)$ and $\beta \in C(\bfa)$. The \emph{rc-graph} for the reduced-word compatible sequence pair $(\bfa, \beta)$, denoted by $D(\bfa, \beta)$, is a diagram obtained by covering the boxes on or above the main anti-diagonal of the $n \times n$ grid with a cross in position $(\beta_i, a_i - \beta_i + 1)$ for $i \in [\ell(w)]$ and a dot in each remaining position.

    \begin{table}[ht]
        \begin{tabular}{|c|c|c|c|}
            \hline
            $i$ & $a_i$ & $\beta_i$ & $(\beta_i, a_i-\beta_i+1)$ \\ \hline \hline
            $1$ & $1$ & $1$ & $(1, 1)$ \\ \hline
            $2$ & $4$ & $2$ & $(2, 3)$ \\ \hline
            $3$ & $3$ & $2$ & $(2, 2)$ \\ \hline
            $4$ & $4$ & $4$ & $(4, 1)$ \\ \hline
        \end{tabular}
        \caption{Coordinates of crosses in the rc-graph $D(\bfa, \beta)$ for $\bfa = 1434$ and $\beta = 1224$.}
        \label{tab:rc-ex}
    \end{table}

\begin{example}\label{ex:RCGraph}
    Consider the permutation $w=[21543] \in S_5$, and the reduced word $\bfa = 1434 \in R(w)$ with $\bfa$-compatible sequence $\beta = 1224$. To construct the rc-graph for $(\bfa, \beta)$, we consider each index $1 \leq i \leq \ell(w) = 4$ and generate the list of coordinate pairs for crosses, as shown on the right in Table \ref{tab:rc-ex}.  We place crosses in these four locations and fill the remaining positions by dots to obtain the rc-graph $D(\bfa, \beta)$ depicted in Figure \ref{fig:RCEx}.
\end{example}

  \begin{figure}[ht]
        \begin{tikzpicture}[scale=0.65]
            \foreach \i in {1,2,3,4,5}{
                \node [align=center] at (0, 6-\i) {\i};
                \node [align=center] at (\i, 6) {\i};
            }
            \foreach \position in {(1, 5), (2, 6-2), (3, 6-2), (1,6-4)}{
                \node [align=center] at \position {$+$};
            }
            \foreach \position in {(1,6-2), (1,6-3), (1,6-5), (2, 6-1), (2,6-3), (2, 6-4), (3,6-1), (3,6-3), (4, 6-1), (4, 6-2), (5, 6-1)}{
                \node [align=center] at \position {$\cdot$};
            }
        \end{tikzpicture}
        \caption{The rc-graph $D(\bfa, \beta)$ for $\bfa = 1434$ and $\beta = 1224$.}
        \label{fig:RCEx}
    \end{figure}

Note in Figure \ref{fig:RCEx} that if we fill the boxes in this rc-graph with elbow tiles rather than dots, then we instead obtain a reduced pipe dream for the permutation $w = [21543]$ such that $\bfa = 1434 \in R(w).$ The construction of the rc-graph thus suggests the following bijection between the set of pipe dreams in $RP(w)$ and compatible sequences $\beta \in C(\bfa)$ for all $\bfa \in R(w)$.

\begin{lemma}\label{lem:wtBijRC}
    Given any $w \in S_n$, consider the map
    \begin{align*}
       \phi_1: RP(w) & \ \longrightarrow \ \bigcup\limits_{\bfa \in R(w)} C(\bfa) \quad \text{defined by} \\
       D_+ = \{(i_1,j_1),\dots,(i_p,j_p)\} & \ \longmapsto \ \beta = (i_1,\dots, i_p) \in C(\bfa)\ \  \text{for}\ \ \bfa = (i_1+j_1-1, \dots, i_p+j_p-1),
    \end{align*}
where the crosses $(i_k,j_k) \in D_+$ are enumerated along the rows of $D$ from top to bottom, recording from right to left within each row. Then $\phi_1$ is a weight-preserving bijection.
\end{lemma}

\begin{remark}\label{rmk:PD=rc}
The formula for $\phi_1$ can be found in \cite{BilleyBergeron}, though there the domain is the set of rc-graphs, which come equipped with a reduced-word compatible sequence pair $(\bfa, \beta)$. Due to the development of $\phi_1$ on rc-graphs in \cite[Section 3]{BilleyBergeron}, and the construction of the word $Q(D):= s_{\bfa} = s_{i_1+j_1-1}\cdots s_{i_p+j_p-1}$ for a pipe dream in \cite[Section 1.4]{KnutsonMiller}, it is common throughout the literature to use the terms rc-graph and pipe dream interchangeably, and to apply $\phi_1$ without comment.
\end{remark}

Although this bijection between the set of rc-graphs and pipe dreams is well-known to experts, we include a proof of Lemma \ref{lem:wtBijRC} both for the sake of completeness, and because our later arguments involving the crystal structure in Section \ref{sec:intertwine} require a fine degree of detail concerning $\phi_1$.   

\begin{proof}[Proof of Lemma \ref{lem:wtBijRC}]
    Fix a permutation $w \in S_n$, and let $D \in RP(w)$. Denote by $p = \ell(w)$, so that $|D_+|=p$. We first show that $\phi_1$ is well-defined, meaning that $\bfa = (i_1+j_1-1, \dots, i_p+j_p-1)$ is a reduced word for $w \in S_n$ and that $\beta = (i_1, \dots, i_p)$ is $\bfa$-compatible. 
    Since each cross $(i,j) \in D_+$ satisfies $i+j\leq n$, then the letter $a_k := i_k+j_k-1 \in [n-1]$ so that $\bfa = a_1\cdots a_p$ is a word for some permutation in $S_n$ of length at most $\ell(w)$. By Lemma 1.4.5 of \cite{KnutsonMiller}, we have $Q(D) = s_{i_1+j_1-1}\cdots s_{i_p+j_p-1} =  [w_1 \cdots w_n] \in S_n$, where the pipe in row $i$ of $D$ exits from column $w_i$.  Moreover, since $|D_+| = \ell(w)$, then $Q(D)$ is a reduced expression, equivalently $\bfa$ is a reduced word for $w$.

    For $\beta = (i_1, \dots, i_p)$, since the rows of $D_+$ are indexed from the top down, then $i_1 \leq \cdots \leq i_p$.  For any $k \in [p]$, since the column index $j_k \geq 1$, then we have $\beta_k = i_k \leq i_k+j_k-1 = a_k$. Finally, suppose that $a_k < a_{k+1}$ for $k \in [p]$, equivalently $ i_k+j_k-1 < i_{k+1}+j_{k+1}-1$. Recall that $i_k \leq i_{k+1}$, and assume for a contradiction that $i_k = i_{k+1}$. Then we have $j_k<j_{k+1}$, which means that the pair $(i_k,j_k),(i_{k+1},j_{k+1}) \in D_+$ lie in the same row $i_k$, but the second cross $(i_{k+1},j_{k+1})$ lies to the right of $(i_k,j_k)$, violating our enumeration of the elements of $D_+$ from right to left within each row.  Therefore, $i_k < i_{k+1}$, completing the proof that $\beta$ is an $\bfa$-compatible sequence.  The map $\phi_1(D)$ thus yields an $\bfa$-compatible sequence for a reduced word $\bfa \in R(w)$ for the given permutation $w \in S_n$ such that $D \in RP(w)$, and so $\phi_1$ is well-defined.

    Conversely, fix $w \in S_n$ and denote by $p = \ell(w)$. Given any $\bfa \in R(w)$ and $\beta \in C(\bfa)$, we claim that the rc-graph $D(\bfa, \beta)$ of \cite{BilleyBergeron}, constructed by placing a cross in boxes $(\beta_k, a_k-\beta_k+1)$ and dots in each remaining position, becomes a reduced pipe dream $\overline{D}(\bfa, \beta) = \overline{D}$ for $w$ if we replace all dots by elbow tiles. First note that for any $k \in [p]$, the box $(\beta_k, a_k-\beta_k+1)$ satisfies $\beta_k+(a_k-\beta_k+1) = a_k+1 \leq n$, and so all crosses lie above the main anti-diagonal. Connecting the crosses and elbows on each tile in the unique possible way, 
    denote by $w_i \in [n]$ the column where the pipe entering row $i$ exits. We claim that the resulting window $[w_1 \cdots w_n]$ equals the permutation $w$ having the given $\bfa$ as a reduced word.

    We proceed by induction on $p = \ell(w)$. If $p=1$, then $\overline{D}(\bfa, \beta)$ has a single cross in position $(\beta_1, a_1-\beta_1+1)$, on the $a_1$ anti-diagonal indexing from northwest to southeast. The only pipes which cross are those starting in rows $i$ and $i+1$, which follow the $a_1$ and $a_1+1$ anti-diagonals, cross each other once at $(\beta_1, a_1-\beta_1+1)$, and then exit from columns $a_1+1$ and $a_1$, respectively.  The window for the corresponding permutation is thus $[1 \cdots a_1+1\; a_1 \cdots n] = s_{a_1}$, concluding the base case.  In general, after removing the cross in position $(\beta_1, a_1-\beta_1+1)$ on the $a_1$ anti-diagonal, \cite[Lemma 3.1]{Kogan} says that we obtain another rc-graph.  Applying the inductive hypothesis, the resulting configuration is a pipe dream for the permutation $s_{a_1}w$ with reduced word $a_2\cdots a_p$. Suppose that the pipes entering rows $i,j$ exit from columns $a_1, a_1+1$, respectively, in the original figure $\overline{D} = \overline{D}(\bfa, \beta)$. In $s_{a_1}w$, we then have $w_i = a_1+1$ and $w_j = a_1$. Adding this cross back in, since it is the first cross using this ordering, the only impact is to swap the values $a_1+1$ and $a_1$ in the window. As left multiplication by $s_{a_1}$ exchanges the values $a_1$ and $a_1+1$ in the window, the window for $w = s_{a_1}(s_{a_1}w)$ thus equals $[w_1 \cdots w_n]$, as required.  Moreover, $|\overline{D}_+| = p = \ell(w)$ by construction, and so $\overline{D} \in RP(w)$ is a indeed a reduced pipe dream for $w$.

    Now given $D \in RP(w)$, it is clear by definition that $\overline{\phi_1(D)} = D$. Conversely, given $\bfa \in R(w)$ and $\beta \in C(\bfa)$, we have $\phi_1(\overline{D}(\bfa, \beta)) = (\bfa, \beta)$ by construction.  Therefore, $\phi_1$ is a bijection.   Finally, recall that $\wt(D)$ is the weak composition whose $i^{\text{th}}$ coordinate is the number of crosses in row $i$ of $D$.  By definition, $\phi_1(D) = (\bfa, \beta)$ where the sequence $\beta = (i_1, \dots, i_p)$ records the row numbers of $D_+$ from top to bottom.  Since the $i^{\text{th}}$ coordinate of $\wt(\beta)$ is the number of $i$'s the sequence $\beta$ contains, then $\wt(D) = \wt(\beta)$, and $\phi_1$ is weight-preserving.
\end{proof}


\section{Reduced Factorizations with Cutoff}\label{sec:rfcs}

This section provides interpretations for all previous combinatorial machinery in terms of certain factorizations of reduced words for a given permutation. In Theorem \ref{thm:RFCSchub}, we state Assaf and Schilling's formula for Schubert polynomials in terms of reduced factorizations satisfying a cutoff condition \cite{assaf_demazure_2018}. We review the underlying weight-preserving bijection to reduced-word compatible sequences in Section \ref{sec:RCtoRFC}.  We also develop the Demazure crystal structure on reduced factorizations with cutoff, summarized from \cite{assaf_demazure_2018} as Theorem \ref{thm:main-rfc}.

\subsection{Reduced factorizations with cutoff}\label{sec:rfc-def}

We first explain how reduced words admit increasing factorizations meeting a cutoff condition, due to \cite{assaf_demazure_2018}.
Given a permutation $w \in S_n$ and any reduced word $\bfa = a_1a_2\cdots a_p \in [n-1]^{\ell(w)}$, a \emph{reduced factorization with cutoff} (RFC) is a partition of $\bfa$ into $n-1$ blocks (possibly empty), such that:
\begin{enumerate}
\item[(i)] subwords within each block are strictly increasing, and
\item[(ii)] the leftmost letter in block $i$ is at least $i$, where blocks are numbered from right to left.
\end{enumerate}
  We denote the set of reduced factorizations with cutoff for all $\bfa \in R(w)$ by $\operatorname{RFC}(w)$.

\begin{example}\label{ex:rfc}
All reduced factorizations with cutoff for \(w = [21543] \in S_5\) are provided in the table below. The reduced words $\bfa \in R(w)$ are listed along the first row, and the factorizations with cutoff appear below in corresponding columns, whenever they exist.

 \begin{center}
		\begin{tabular}{cccccccc}
			1343 & 1434 & 3143 & 3413 & 3431 & 4134 & 4314 & 4341 \\ \hline 
			$\emptyset$ & $\emptyset$ & $\emptyset$ & \( ( \;  )( \;  )( 34  )( 13 ) \) & \( (\; )( 34  )( 3  )( 1 ) \) & \( ( \;  )( \;  )( 4  )( 134 ) \) & \( ( \;  )( 4  )( 3  )( 14  ) \) & \( ( \;  )( 4  )( 34  )( 1 ) \)  \\
			     &     &      &   \( ( \;  )( 3  )( 4  )( 13 ) \) &      & \( ( \;  )( 4  )( \;  )( 134 ) \) &  \( ( 4  )( \;  )( 3  )( 14 ) \) & \( ( 4  )( \;  )( 34  )( 1 ) \) \\
			     &    &      &   \( ( \;  )( 34  )( \;  )( 13 ) \) &      & \( ( 4  )( \;  )(\;  )( 134 ) \) & \( ( 4  )( 3  )( \;  )( 14 ) \) & \( ( 4  )( 3  )( 4  )( 1 ) \) \\
			     &    &     &         &     &       &      & \( ( 4  )( 34  )( \;  )( 1 ) \)
        \end{tabular}
	\end{center}
\end{example}

The \emph{weight of a reduced factorization} $r \in RFC(w)$, denoted $\wt(r)$, is the weak composition of $\ell(w)$ whose $i^{\text{th}}$ coordinate equals the number of letters in the $i^{\text{th}}$ block of $r$.  For example, if $r = (4)(3)(\;)(14) \in RFC([21543])$, then $\wt(r) = (2,0,1,1)$.

Comparing Examples \ref{ex:compatible} and \ref{ex:rfc}, one anticipates that the weight function defines a natural bijection between the $\bfa$-compatible sequences for $\bfa \in R(w)$ and the increasing factorizations with cutoff in $RFC(w^{-1})$; see Lemma \ref{lem:wtBijection} for details. The corresponding formula for Schubert polynomials as generating functions over RFCs is due to Assaf and Schilling \cite{assaf_demazure_2018}.

\begin{theorem}[Proposition 5.5 \cite{assaf_demazure_2018}] \label{thm:RFCSchub}
Let $w \in S_n$. Then
\begin{equation*}
    \mathfrak{S}_w(x_1, \dots, x_n) = \sum_{r \in RFC(w^{-1})}  \mathbf{x}^{\wt(r)}.
\end{equation*}
\end{theorem}

\noindent Note by comparing Theorem \ref{thm:RCSchub} that reduced factorizations with cutoff simultaneously track the weight, equivalently, the compatible sequence, as well as the reduced word.

\subsection{Compatible sequences and RFCs}\label{sec:RCtoRFC} 

We now review the explicit relationship between compatible sequences and reduced factorizations with cutoff presented in  \cite{assaf_demazure_2018}, which is at the heart of Theorem \ref{thm:RFCSchub}, as well as our proof of Theorem \ref{thm:main}.

\begin{lemma}\label{lem:wtBijection}
    Given any $w \in S_n$, consider the map 
    \begin{align*}
       \phi_2: \bigcup\limits_{\bfa \in R(w)} C(\bfa) & \ \longrightarrow\ RFC(w^{-1})  \quad \text{defined by} \\
       \beta=\beta_1\cdots \beta_p & \ \longmapsto \ r = (r^{n-1})(\cdots)(r^1)\ \ \text{such that}\ \  \beta_j =i \Longleftrightarrow a_j \in (r^i),
    \end{align*}
where $\beta$ is an $\bfa$-compatible sequence for $\bfa = a_1\cdots a_p \in R(w)$.  Then $\phi_2$ is a weight-preserving bijection.
\end{lemma}

\begin{proof}
This result is contained within the proof of \cite[Proposition 5.5]{assaf_demazure_2018}. 
\end{proof}

We illustrate the bijection $\phi_2$ from Lemma \ref{lem:wtBijection} in the following example.

\begin{example}
    Let $w = [21543] \in S_5$, and consider the reduced word $\bfa = 3143 \in R(w)$. Recall from Example \ref{ex:compatible} that $\beta = 1122$ is an $\bfa$-compatible sequence. Since $\beta = 1122$ contains two 1's and 2's, then $\wt(\beta) = (2,2,0,0)$. From the pair $(\bfa, \beta)$, we construct the reduced factorization prescribed by Lemma \ref{lem:wtBijection} as follows.

    Since $\beta_1 = \beta_2 = 1$, then we place $a_1=3, a_2=1$ in block $(r^1)$, recording the letters in increasing order so that the result is an increasing factorization.  Similarly, since $\beta_3 = \beta_4 = 2$, then we place $a_3, a_4 \in (r^2)$. Altogether, we have that $\phi(\beta) = r= (\;)(\;)(34)(13)$. Note that the resulting factorization records the entries of $\bfa$ in reverse order, so that $r \in RFC(w^{-1})$. Finally, $\wt(r) = (2,2,0,0) = \wt(\beta)$, since blocks are numbered from right to left.
\end{example}

\subsection{Crystal operators on RFCs}\label{sec:RFC-crystal}

In this section, we review a family of operators defined on the set of reduced factorizations with cutoff for a given permutation, shown by Assaf and Schilling in  \cite{assaf_demazure_2018} to produce a Demazure crystal structure.

\begin{definition}\label{def:pair-rfc}
Given $w \in S_n$ and a reduced factorization $r \in RFC(w)$, write $r = (r^{n-1})(\cdots)(r^2)(r^1)$, numbering its blocks from right to left.  Fix a block index $1 \leq i <n$, and denote the largest, equivalently rightmost, letter in block $(r^i)$ by $a$. Define a \emph{pairing process on block} $(r^i)$ as follows:
\begin{enumerate}
     \item Look for the smallest unpaired letter $b$ in block $(r^{i+1})$ such that $a<b$.
      \begin{enumerate}
            \item If such a letter $b$ exists, we say that $a$ and $b$ are \emph{paired}.
            \item If no such $b$ exists, we say that $a$ is \emph{unpaired}.
        \end{enumerate}
    \item Denote by $a'$ the letter in block $(r^i)$ which is immediately to the left of $a$.
     \begin{enumerate}
            \item If such an $a'$ exists, we reset $a := a'$ and start again from step (1).
            \item If no such $a'$ exists, the pairing process on block $(r^i)$ is complete.
        \end{enumerate}
        \end{enumerate}
\end{definition}

We invite the reader to compare the following example illustrating the pairing process on reduced factorizations with the pairing process on pipe dreams from Example \ref{ex:pairing}.

\begin{example}\label{ex:pairing-rfc}
    We illustrate the pairing process on $r = (\;)(\;)(4)(134) \in RFC(w)$ for $w = [21543]$. Fix $i=1$ and identify $a=4$ as the largest letter in block $(r^1) = (134)$. We now identify the smallest unpaired letter $b=4$ in block $(r^2) = (4)$. Since $a=4 \not< 4=b$, then $a=4$ is unpaired, and we color $a$ red in $r = (\;)(\;)(4)(13\red{4})$. 

    In step (2), we identify $a' = 3$ as the letter in block $(r^1) = (13\red{4})$ immediately to the left of $a=4$.  We thus return to step (1) applied to $a=3$ instead. Here again the smallest unpaired letter in block $(r^2) = (4)$ is $b=4$, which now satisfies $a=3<4=b$. Therefore, these letters get paired in step (1a), as we indicate by coloring them green in $r = (\;)(\;)(\green{4})(1\green{3}\red{4})$.

    Finally, we return to step (2) to identify the only remaining letter in $(r^1) = (1\green{3}\red{4})$ as $a'=1$ and reset $a=1$. As there are no unpaired letters in block $(r^2) = (\green{4})$, then $a$ remains unpaired in $r = (\;)(\;)(\green{4})(\red{1}\green{3}\red{4})$. The pairing process on block $i=1$ is now complete, having analyzed all letters in $(r^1)$.
\end{example}

After running the pairing process on block $(r^i)$ of $r \in RFC(w)$, we can define operators $e_i$ and $f_i$, which produce another element of $RFC(w)$ whenever they are nonzero. Our treatment below slightly rephrases the development of these operators on reduced factorizations from \cite[Section 3.2]{MorseSchilling}, refined for factorizations meeting the cutoff condition in \cite[Section 4.3]{assaf_demazure_2018}.

\begin{definition}\label{def:lowering-rfc}
    Let $r \in RFC(w)$ for some $w \in S_n$. Fix an $1 \leq i < n$ and run the pairing process on block $(r^i)$ of $r$.
    \begin{enumerate}
    \item If all letters in $(r^i)$ are paired, then set $f_i(r)=0$. Otherwise, denote by $u \in (r^i)$ the smallest unpaired letter.  Denote by
    $t = \max \{ z \le u \mid z - 1 \not\in (r^i)\}.$
    The \emph{lowering operator} $f_i$ on $r$ is defined by $f_i(r) = (r^{n-1})\cdots (\tilde{r}^{i+1})(\tilde{r}^i)\cdots (r^1)$, where
    \[(\tilde{r}^i) = (r^i) \backslash \{ u\} \quad \text{and} \quad (\tilde{r}^{i+1}) = (r^{i+1}) \cup \{t\}. \]
    If $f_i(r) \notin RFC(w)$, then set $f_i(r)=0$. 
    
    \item If all letters in $(r^{i+1})$ are paired, then set $e_i(r)=0$. Otherwise, denote by $v \in (r^{i+1})$ the largest unpaired letter. Denote by $s = \min\{ z \geq v \mid z+1 \notin (r^{i+1}) \}$. The \emph{raising operator} $e_i$ on $r$ is defined by $e_i(r) = (r^{n-1})\cdots (\tilde{r}^{i+1})(\tilde{r}^i)\cdots (r^1)$, where
    \[(\tilde{r}^i) = (r^i) \cup \{ s\} \quad \text{and} \quad (\tilde{r}^{i+1}) = (r^{i+1}) \backslash \{v\}. \]
    If $e_i(r) \notin RFC(w)$, then set $e_i(r)=0$. 
    \end{enumerate}
\end{definition}

We illustrate these operators on RFCs in an example below.  
Compare Definitions \ref{def:crystalchutes} and \ref{def:inversechutes} which defined the corresponding lowering and raising operators on pipe dreams.

\begin{example}\label{ex:lower-rfc}
   Consider $r = (\;)(4)(3)(14) \in RFC(w)$ for $w = [21543]$.  The outcome of the pairing process on block $i=1$ is $r = (\;)(4)(\green{3})(\green{1}\red{4})$, where green (resp.~red) indicates paired (resp.~unpaired) letters.  Since there are no unpaired letters in block $(r^2) = (\green{3})$, then $e_1(r) = 0$.
   
   To calculate $f_1(r)$, note that $u=4$ is the smallest unpaired element in block $(r^1)$, and define $t = \min \{ v \le 4 \mid v - 1 \not\in (r^1)\}$. Since $4-1=3 \notin (r^1) = (14)$, then $t=4$. Therefore, the lowering operator $f_1$ on $r$ deletes the $u=4$ from block $(r^1)$ and inserts $t=4$ into block $(r^2)$, giving $f_1(r) = ()(4)(34)(1)$.
\end{example}

\begin{definition}
    Let $r \in RFC(w)$ for $w\in S_n$.
     If $e_i(r)=0$ for all $1 \leq i <n$, then we refer to $r$ as a \emph{highest weight factorization}. If $f_i(r)=0$ for all $1 \leq i <n$, then we refer to $r$ as a \emph{lowest weight factorization}. 
\end{definition}

In \cite{MorseSchilling}, Morse and Schilling prove that the operators $e_i$ and $f_i$ for $1 \leq i < n$, together with the weight function $\wt(r)$, define a type $A_{n-1}$ crystal structure on the set of all reduced factorizations for $w \in S_n$.  The additional cutoff condition then corresponds to a Demazure crystal truncation as follows, shown by Assaf and Schilling \cite{assaf_demazure_2018}.

\begin{theorem}[Theorem 5.11 \cite{assaf_demazure_2018}]\label{thm:main-rfc} 
Given any $w \in S_n$, the operators $e_i$ and $f_i$ for $1 \leq i < n$ define a type $A_{n-1}$ Demazure crystal structure on $RFC(w)$. That is, 
\begin{equation*}
    RFC(w) \cong \bigcup\limits_{\substack{r \in RFC(w) \\ e_i(r) = 0,\; \forall 1 \leq i < n}} B_{\pi_r}(\wt(r)),
\end{equation*}
where $\pi_r = [\pi_1\; \cdots\; \pi_n] \in S_n$ is the shortest permutation that sorts the weight $a = (a_1, \dots, a_n)$ of the weak Edelman-Greene insertion tableau for $r$, in the sense that $\wt(r) = (a_{\pi_1}, \dots, a_{\pi_n})$; see Section \ref{sec:permutation} for details on this insertion algorithm.
\end{theorem}


\section{Intertwining Pipe Dreams and RFCs}\label{sec:intertwine}

In this section, we combine results from the previous two sections to obtain a map directly from pipe dreams to reduced factorizations with cutoff; see Proposition \ref{prop:wtBijPD} and Corollary \ref{cor:PDtoRFCwhole}. We then prove that this weight-preserving bijection also respects all required crystal structures, such as the pairing process and lowering operators, in order to prove Theorem \ref{thm:main} and its immediate corollary Theorem \ref{thm:main-key}.

\subsection{Weight-preserving bijection from pipe dreams to RFCs}

To relate the crystal structure on $RFC(w)$ from Theorem \ref{thm:main-rfc} to the proposed crystal structure on $RP(w)$ in Theorem \ref{thm:main}, we first require a weight-preserving bijection between the two underlying sets.  Composing the bijection of Lemma \ref{lem:wtBijRC} from pipe dreams to compatible sequences with the bijection to reduced factorizations with cutoff provided by Lemma \ref{lem:wtBijection} yields the following weight-preserving bijection directly from $RP(w)$ to $RFC(w^{-1})$.

\begin{proposition}\label{prop:wtBijPD}
    Given any $w \in S_n$, the map
    \begin{align*}
       \phi: RP(w) & \ \longrightarrow\ RFC(w^{-1})  \quad \text{defined by} \\
       (i,j) \in D_+ & \ \longmapsto \ i+j-1 \in (r^i)
    \end{align*}
is a weight-preserving bijection.
\end{proposition}

\begin{proof}
   We claim that $\phi = \phi_1 \circ \phi_2$, where $\phi_1$ and $\phi_2$ are the weight-preserving bijections from Lemmas \ref{lem:wtBijRC} and \ref{lem:wtBijection}, respectively. Let $(i,j) \in D_+$, and suppose that $(i,j)$ is the $k^{\text{th}}$ cross, enumerating the rows of $D$ from top to bottom, recording from right to left within each row, as in Lemma \ref{lem:wtBijRC}. Then $\phi_1(D)$ builds a reduced word $\bfa = a_1 \cdots a_p \in R(w)$ such that $a_k = i+j-1$, with compatible sequence $\beta \in C(\bfa)$ such that $\beta_k = i$. Now applying $\phi_2$ from Lemma \ref{lem:wtBijection}, we place $a_k = i+j-1 \in (r^i)$ since $\beta_k = i$. Therefore, $\phi(i,j) = \phi_2\circ \phi_1(i,j)$ for all $(i,j) \in D_+$, and so  $\phi = \phi_1 \circ \phi_2$ and is thus also a weight-preserving bijection. 
\end{proof}

The bijection $\phi$ is presented in Proposition \ref{prop:wtBijPD} as a correspondence between individual crosses in a pipe dream and single letters in the associated reduced factorization.  Putting this information together yields a convenient algorithm for going directly between the full graphical representation of a pipe dream and the entire reduced factorization.

\begin{corollary}\label{cor:PDtoRFCwhole}
Let $w \in S_n$.  Given any $D \in RP(w)$, the image $\phi(D)$ is obtained as follows:
\begin{enumerate}
    \item Shift all crosses in $D_+$ in row $i$ to the right by $i-1$ columns to obtain $D'$.
    \item Define block $(r^i)$ by reading the columns of crosses in row $i$ of $D_+'$ from left to right.
    \item Then $\phi(D) = (r^{n-1}) \cdots (r^1)$.
\end{enumerate}
\end{corollary}

\begin{proof}
    This algorithm follows directly from the formula provided in Proposition \ref{prop:wtBijPD}.
\end{proof}

We provide an example to explain Corollary \ref{cor:PDtoRFCwhole} below. 

\begin{example}
Let $w = [21543]$, and consider $D_+ =\{(1,1),(1,4),(2,2),(3,2)\}$. In step (1) of constructing $\phi(D)$, we shift the cross in row 2 to the right by 1, and the cross in row 3 to the right by 2. In step (2), we build block $(r^1)$ by reading off the column indices $(14)$ from row 1. Looking at column indices of the shifted crosses, we build block $(r^2)$ by recording the shifted column index $(3)$, and similarly, $(r^3) = (4)$. Altogether, $\phi(D) = (\;)(4)(3)(14)$ by Corollary \ref{cor:PDtoRFCwhole}.
\end{example}

\subsection{Intertwining RFCs and pipe dreams}

Beyond having a weight-preserving bijection $\phi$ between the set of pipe dreams and reduced factorizations with cutoff for a given permutation, to verify the Demazure crystal structure claimed in Theorem \ref{thm:main}, we also require that $\phi$ preserves the underlying crystal operations.  We show in Lemma \ref{lem:wtPair} that $\phi$ respects the pairing process, and finally that $\phi$ preserves the raising and lowering operators in Proposition \ref{prop:intertwine}. We conclude with the proof of the Demazure crystal decomposition in Theorem \ref{thm:main}.

\begin{lemma}\label{lem:wtPair}
    Given any $w \in S_n$, the bijection $\phi: RP(w) \to RFC(w^{-1})$
preserves the pairing process on row $i$ of $D \in RP(w)$, respectively block $(r^i)$ of $r \in RFC(w^{-1})$, for all $1 \leq i < n$.
\end{lemma}

\begin{proof}
     We proceed by induction on the number of crosses in row $i$ of $D \in RP(w)$. As the base case, suppose there are no crosses in row $i$, in which case nothing is paired when we run the pairing process on row $i$ of $D$. Since $\phi$ bijectively maps crosses in row $i$ to letters in block $(r^i)$ of $r = \phi(D)$, then $(r^i)$ is empty and the pairing process on block $(r^i)$ of $r$ also pairs nothing.

     Now assume that we have analyzed $\ell$ crosses in the pairing process, and let $c = (i,j)$ denote the $(\ell+1)^{\text{st}}$ cross from the right in row $i$ of $D$. We first suppose that there exists a leftmost cross $c_+ = (i+1,k)$ for $k \geq j$ minimal such that $c$ pairs with $c_+$. By Proposition \ref{prop:wtBijPD}, we have $\phi(c) = i+j-1 \in (r^i)$ and $\phi(c_+) = i+k \in (r^{i+1})$. Since $j \leq k$, then note that $i+j-1 < i+k$. Moreover, no smaller unpaired letter $b \in (r^{i+1})$ such that $i+k-1 < b$ exists, since otherwise there would be a cross farther left in row $i+1$ to pair with $c$ than $c_+$ by Proposition \ref{prop:wtBijPD}. Therefore, $\phi(c)$ and $\phi(c_+)$ are paired in the pairing process on block $(r^i)$ of $r = \phi(D)$.

     If instead $c=(i,j)$ is unpaired in step $\ell+1$ of the pairing process on row $i$ of $D$, then there does not exist an unpaired cross in row $i+1$ which lies weakly right of $c$.  Denote by $a = \phi(i,j) = i+j-1$, the rightmost letter in block $(r^i)$ of $r = \phi(D)$ which remained to be analyzed in the pairing process. If there were a letter $b \in (r^{i+1})$ such that $a<b$, then $i+j-1<b$ and so $j \leq b-i$. On the other hand, $\phi^{-1}(b) = (i+1, b-i)$ by Proposition \ref{prop:wtBijPD}, which is yet unpaired by the inductive hypothesis. Since $b-i\geq j$, then $(i+1,b-i)$ is an unpaired cross in row $i+1$ which lies weakly right of $c =(i,j)$. Since $c$ is unpaired, however, then $b$ cannot exist, and $\phi(c)$ is unpaired as well.

     In summary, when crosses $c$ and $c_+$ are paired in $D$, then letters $\phi(c)$ and $\phi(c_+)$ are paired in $\phi(D)$. Similarly, if $c$ is unpaired in $D$, then $\phi(c)$ remains unpaired in $\phi(D)$. The reverse argument, starting instead with the pairing process on letters in $r = \phi(D)$, follows via the bijection $\phi^{-1}$.
\end{proof}

Our next goal is to prove that the bijection $\phi$ intertwines with the operators $e_i$ and $f_i$. In particular, we must first show that these operators are zero on corresponding elements under $\phi$.

\begin{lemma}\label{lem:zero}
    Let $w \in S_n$. Given $D \in RP(w)$ and any $1 \leq i <n$, we have \[f_i(D) = 0 \quad \Longleftrightarrow \quad f_i(\phi(D)) = 0.\]
\end{lemma}

\begin{proof}
    Let $D \in RP(w)$, and denote by $r = \phi(D)$. First suppose that $f_i(D) = 0$ for some $1 \leq i < n$. By Definition \ref{def:crystalchutes}, either every cross in row $i$ is paired after running the pairing process on row $i$, or every tile to the left of the leftmost unpaired cross in row $i$ is another cross.  Since $\phi$ maps crosses in row $i$ to block $(r^i)$ of $r$ and preserves the pairing by Lemma \ref{lem:wtPair}, then in the first case, all letters in block $(r^i)$ are also paired, and $f_i(r)=0$ as well by Definition \ref{def:lowering-rfc}. 
    
    In the second case, denote by $(i,j) \in D_+$ the leftmost unpaired cross in row $i$. By hypothesis, $(i,k) \in D_+$ for all $1 \leq k <j$.  Since $(i,k) \in D_+$ for all $1 \leq k \leq j$, then we have $(r^i) = (1+i-1, 2+i-1,\dots, j+i-1, \dots) = (i,i+1,\dots, i+j-1, \dots)$. By Lemma \ref{lem:wtPair}, the smallest unpaired letter in block $(r^i)$ equals $\phi(i,j) = i+j-1$. Therefore, $t = \max \{z \leq i+j-1 \mid z-1 \notin (r^i)\}.$ Note that for all $i+1 \leq z \leq i+j-1$, we have $z-1 \in (r^i)$. Therefore, $t = i$ and $f_i(r)$ inserts $i$ into block $(r^{i+1})$, violating the cutoff condition.  By Definition \ref{def:lowering-rfc}, since $f_i(r) \notin RFC(w^{-1})$, we have $f_i(r) = 0$ in this case as well.

    Conversely, suppose that $f_i(r) = 0$ where $r \in RFC(w^{-1})$ satisfies $r = \phi(D)$ for some $D \in RP(D)$. By Definition \ref{def:lowering-rfc}, either all the letters in $(r^i)$ are paired, or $f_i(r) \notin RFC(w^{-1})$. In the first case, every cross in row $i$ is also paired by $\phi$, and so $f_i(D)=0$ as well. Now suppose there exists an unpaired letter in $(r^i)$, but $f_i(r) \notin RFC(w^{-1})$. By \cite[Theorem 4.10]{assaf_demazure_2018}, we know that $f_i(r)$ is actually a reduced factorization for $w^{-1}$, so the only possibility is that applying $f_i$ to $r$ violates the cutoff condition. As $f_i$ only inserts one new letter $t$ into block $(r^{i+1})$, it must be the case that $t=i$. By the same argument as above, under $\phi$ we have that $(i,k) \in D_+$ for all $1 \leq k <j$ where $(i,j)$ is the leftmost unpaired cross in row $i$. Therefore, $f_i(D) = 0$ by Definition \ref{def:crystalchutes}.
\end{proof}

We are now able to prove that $\phi$ preserves the lowering operators; equivalently,  $\phi$ intertwines with the crystal chute moves $f_i$.

\begin{proposition}\label{prop:intertwine}
    Given any \(w \in S_n\), the bijection
	$\phi: RP(w) \to RFC(w^{-1})$ 
preserves the lowering operators; that is, for any $D \in RP(w)$, we have $\phi(f_i(D)) = f_i(\phi(D))$
for all $1 \leq i < n$.
\end{proposition}

\begin{proof}
    Let $D \in RP(w)$ and $1 \leq i <n$.  Recall by Lemma \ref{lem:zero} that the operator $f_i$ is zero on $D$ and $\phi(D)$ simultaneously, so we may extend the bijection by assigning $\phi(0)=0$ in a manner which preserves the raising and lowering operators.

    Now let $D \in RP(w)$, and fix any index $1 \leq i < n$ such that $f_i(D) \neq 0$. In the notation of Definition \ref{def:crystalchutes},  $f_i(D)_+ = D_+ \backslash \{ (i,j) \} \cup \{(i+1, j-m)\}$, where $(i,j)\in D_+$ is the leftmost unpaired cross in row $i$ after running the pairing process on row $i$ of $D$.

    The cross $(i,j)$ that is removed in applying $f_i$ to $D$ corresponds to the letter $\phi(i,j) = i+j-1 \in (r^i)$ in the $i^{\text{th}}$ block of $r = \phi(D) \in RFC(w^{-1})$ via Proposition \ref{prop:wtBijPD}. By Lemma \ref{lem:wtPair}, we know that $i+j-1$ is also the smallest unpaired letter in block $(r^i)$. Therefore, $u =i+j-1$ in the notation of Definition \ref{def:lowering-rfc}, and so applying $f_i(r)$ removes the letter $u = i+j-1 = \phi(i,j)$ from block $(r^i)$.

    Now consider the cross that is added in applying $f_i$ to $D$; namely $(i+1,j-m) \notin D_+$.  We must show that $\phi(i+1,j-m) = (i+1)+(j-m)-1 = i+j-m$ is the letter which gets inserted into block $(r^{i+1})$ when applying $f_i(r)$.  Recall from Definition \ref{def:lowering-rfc} that $t = \max \{z \leq i+j-1 \mid z-1 \notin (r^i) \}.$ As $m \in \N$, we have $i+j-m \leq i+j-1$.  By property (a) in Definition \ref{def:crystalchutes}, position $(i,j-m)$ of $D$ is required to be an elbow tile. By Proposition \ref{prop:wtBijPD}, we thus have $\phi(i,j-m) = (i+j-m)-1 \notin (r^i)$, and so $t \geq i+j-m$. 
    
    If $m=1$, then since we also have $t \leq i+j-1$, we see that $t=i+j-m$ in this case. If $m >1$, then for any $1 \leq k < m$, the cross in position $(i,j-k) \in D_+$ guaranteed by property (b) of Definition \ref{def:crystalchutes} corresponds to $\phi(i,j-k) = i+j-k-1 \in (r^i)$ via Proposition \ref{prop:wtBijPD}. Therefore, $i+j-m =  \max \{z \leq i+j-1 \mid z-1 \notin (r^i) \}$ in case $m > 1$ as well. Applying $f_i(r)$ thus inserts the letter $t = i+j-m = \phi(i+1,j-m)$ into block $(r^{i+1})$. 
    
    Altogether, we have shown that corresponding crosses and letters under $\phi$ are added and removed when applying $f_i$ to pipe dreams or reduced factorizations with cutoff, and so $\phi(f_i(D)) = f_i(\phi(D))$ whenever the operator $f_i$ is nonzero.
\end{proof}

From the previous several results on lowering operators, we can derive the analogous statements for raising operators.

\begin{corollary}\label{cor:intertwine-e}
    Let \(w \in S_n\). For any $D \in RP(w)$ and any $1 \leq i < n$, we have 
\[\phi(e_i(D)) = e_i(\phi(D)) \quad \text{and} \quad e_i(D) = 0 \ \Longleftrightarrow \ e_i(\phi(D)) = 0.\] 
\end{corollary}

\begin{proof}
    First suppose that $e_i(D) \neq 0$, and denote by $r = \phi(D)$. By Lemma \ref{lem:inversechutes}, there exists $D' \in RP(w)$ such that $D = f_i(D')$. Therefore, $\phi(e_i(D)) = \phi(e_i(f_i(D'))) = \phi(D')$ by Proposition \ref{prop:e_iwelldef}. On the other hand, $e_i(\phi(D)) = e_i(r) = e_i(f_i(r')) = r'$ for some $r' \in RFC(w^{-1})$,
    by the inverse relationship of the operators $e_i$ and $f_i$ on RFCs. To prove the first equality, it thus suffices to verify that $\phi(D') = r'$. Note that $\phi(D) = \phi(f_i(D')) = f_i(\phi(D'))$ by Proposition \ref{prop:intertwine}, and write $\phi(D) = r = f_i(r')$.  Therefore, $f_i(r') = f_i(\phi(D'))$, and applying $e_i$ to both sides recovers $r' = \phi(D')$, as required to conclude that $\phi(e_i(D)) = e_i(\phi(D))$.  In particular, both $\phi(e_i(D))$ and $e_i(\phi(D))$ are nonzero whenever $e_i(D) \neq 0$ and vice versa, proving the second claim.
\end{proof}

We are now prepared to prove our main theorems, apart from the formula for the truncating permutation $\pi_D$, equivalently the composition $a_D$, which we prove separately as Theorem \ref{thm:AlgCor}.

\begin{proof}[Proof of Theorem \ref{thm:main}]
    The bijection $\phi: RP(w) \to RFC(w^{-1})$ intertwines with both the raising and lowering crystal chute moves by Proposition \ref{prop:intertwine} and Corollary \ref{cor:intertwine-e}, and it preserves the weight function by Proposition \ref{prop:wtBijPD}. The crystal axioms and Demazure crystal structure thus follow directly from the corresponding result \cite[Theorem 5.11]{assaf_demazure_2018} on RFCs. 
    
    See Algorithm \ref{alg:EG} for the definition of the diagram $\widetilde{D}$ and Theorem \ref{thm:AlgCor} for the claim that $\pi_D$ is the shortest permutation such that $\wt(\widetilde{D}) = \pi_D(\wt(D))$, completing the proof of Theorem \ref{thm:main}.
\end{proof}

\begin{proof}[Proof of Theorem \ref{thm:main-key}]
    This is an immediate corollary of Theorem \ref{thm:main}.
\end{proof}


\section{Permutation indexing the Demazure Crystal}\label{sec:permutation}

This section explains how to identify the truncating permutation $\pi_D$ from a highest weight pipe dream $D$. We begin by describing how to obtain a new diagram $\widetilde{D}$ from $D$ via Algorithm \ref{alg:EG}. The remainder of the paper is devoted to proving that the weight of this new diagram $\widetilde{D}$ indexes the key polynomial corresponding to the pair $(\pi_D, \wt(D))$ from Theorem \ref{thm:main-key}.

\subsection{Truncating permutation from a highest weight pipe dream}

We now define a process for determining the truncating permutation $\pi_D$ from any $D \in RP(w)$ such that $D$ is highest weight; see Theorem \ref{thm:AlgCor}. The outcome produces a shape that can be described by a weak composition, given by the following algorithm.

\begin{alg}\label{alg:EG}
Let $D$ be a highest weight pipe dream.
\begin{enumerate}
    \item For each cross in row $i$, shift it to the right by $i-1$.
    \item For each row, beginning in the lowest row, move the leftmost cross down to the row such that its row and column index match. Fix these crosses.
    \item Set $\ell = 2$. 
        \begin{enumerate}
        \item Beginning at the bottom row containing unfixed crosses, consider the leftmost unfixed cross. Move that cross down to the lowest possible row, remaining in its current column, such that:
        \begin{enumerate}
            \item The cross may not move through other crosses;
            \item The cross is the $\ell^\text{th}$ cross from the left in its new row; and
            \item The cross does not have any previously fixed crosses to its right in the new row.
        \end{enumerate}
        \item Fix this moved cross.
        \item Repeat steps (a) and (b) once within each row until all rows with unfixed crosses have been considered.
    \end{enumerate}
    \item Increment $\ell$ by 1 and repeat step (3).
\end{enumerate}
Once all crosses in the diagram are fixed, the algorithm terminates. Denote the resulting diagram by $\widetilde{D}$.
\end{alg}

The weight of the resulting diagram $\widetilde{D}$ is the weak composition whose $i^{\text{th}}$ coordinate equals the number of crosses in row $i$ of $\widetilde{D}$. We first illustrate Algorithm~\ref{alg:EG} on an example; see Proposition \ref{prop:AlgWD} below for well-definedness.

\begin{example}\label{ex:alg}
    Consider the permutation $w = [4726315] \in S_7$. One of its highest weight pipe dreams $D$ is depicted in Figure \ref{fig:highestWeights}.      The result after applying steps (1) and (2) of Algorithm \ref{alg:EG} to $D$ is shown in Figure \ref{fig:crossesShifted}, with fixed crosses displayed in red. 

\begin{figure}[ht]
\centering
\begin{minipage}{.5\textwidth}
  \centering
  \includegraphics[width=.8\linewidth]{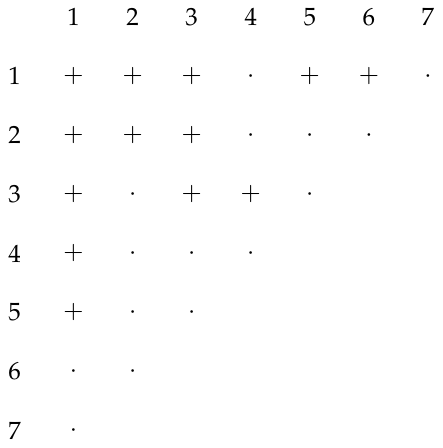}
  \captionof{figure}{A highest weight pipe dream $D$ for the permutation $w = [4726315] \in S_7$.}
  \label{fig:highestWeights}
\end{minipage}%
\begin{minipage}{.5\textwidth}
  \centering
  \includegraphics[width=.8\linewidth]{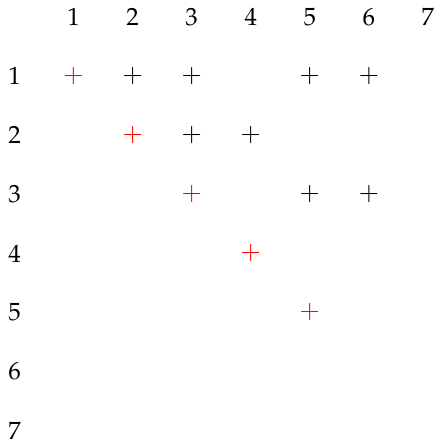}
  \captionof{figure}{The result of applying steps (1) and (2) of Algorithm \ref{alg:EG} to $D$; fixed crosses are red.}
  \label{fig:crossesShifted}
\end{minipage}
\end{figure}

     We now move to the iterative step (3). Set $\ell = 2$. We begin on the lowest row with an unfixed cross, that being row $3$. We move the leftmost unfixed cross in this row, that being the cross at $(3,5)$, down in its column to a position that meet criteria (i) through (iii). We first observe that there is a cross at $(5,5)$, meaning that we are unable to move our cross to row $5$ or any row below it. Our only option is to move this cross to row $4$. Observe that a cross at $(4,5)$ would be the second cross in its row. Thus, we move the cross at $(3,5)$ to $(4,5)$ and fix it there.

    The two crosses at $(2,3)$ and $(1,2)$ cannot move lower without violating (i).  These two crosses are thus also fixed, completing the round of moves for $\ell=2$.
     At the end of this round, we obtain the diagram shown in Figure \ref{fig:ellIs2}. We then increment $\ell$ to $3$, and repeat the process. The final result $\widetilde{D}$ is shown in Figure \ref{fig:ellIs5}.

\begin{figure}[h]
\centering
\begin{minipage}{.5\textwidth}
  \centering
  \includegraphics[width=.8\linewidth]{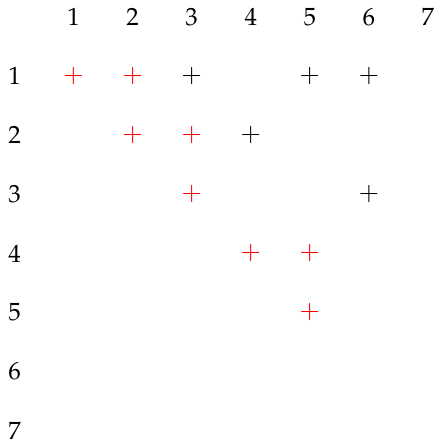}
  \captionof{figure}{The diagram after completing the first iteration of step (3); fixed crossed are red.}
  \label{fig:ellIs2}
\end{minipage}%
\begin{minipage}{.5\textwidth}
  \centering
  \includegraphics[width=.8\linewidth]{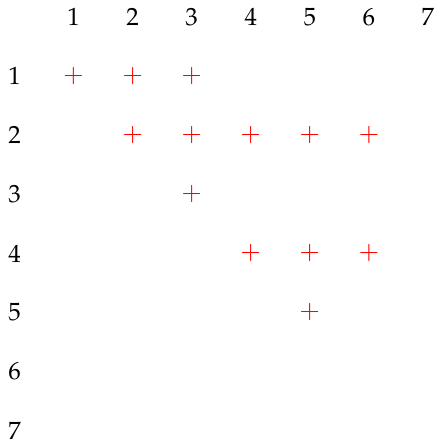}
  \captionof{figure}{The diagram $\widetilde{D}$ after completing Algorithm \ref{alg:EG}, with resulting $\wt(\widetilde{D}) = (3,5,1,3,1,0)$.}
  \label{fig:ellIs5}
\end{minipage}
\end{figure}

\end{example}

Our next goal is to prove that Algorithm \ref{alg:EG} is well-defined. We first record a helpful lemma that will be used several times within the argument.

\begin{lemma}\label{lem:HWPDstep1}
     Let $D \in RP(w)$ be a highest weight pipe dream. Let $c = (i,j^i_\ell) \in D_+$ denote the $\ell^{\text{th}}$ cross from the left in row $i$ of $D$. Then for any fixed $1 \leq \ell \leq \wt(D)_1$, the shifted column sequence $(j^i_\ell + i-1)$ is strictly increasing.
\end{lemma}

\begin{proof}
   Since the column sequence $(j^i_\ell)$ is weakly increasing by Lemma \ref{lem:HWPDincseq}, after all crosses in row $i$ are shifted to the right by $i-1$, the resulting column sequence $(j^i_\ell + i-1)$ is strictly increasing.
\end{proof}

\begin{proposition}\label{prop:AlgWD}
    Let $D \in RP(w)$ be a highest weight pipe dream for some $w \in S_n$.  The diagram $\widetilde{D}$ from Algorithm \ref{alg:EG} is well-defined. Moreover, $\wt(\widetilde{D}) = \pi(\wt(D))$ for some permutation $\pi \in S_n$.
\end{proposition}

\begin{proof}
    Let $D \in RP(w)$ be a highest weight pipe dream. Throughout the proof, given a cross $c$, denote by $\tilde{c}$ the image of $c$ after it is fixed by the algorithm, and denote by $c'$ the closest cross to the left of $c$ in its original row. Denote by $\tilde{c}' = (\tilde{c})'$ the cross to the immediate left of $\tilde{c}$ in its original row. Denote the row index of $c$ recorded from the top down as $\row(c)$.

   By Lemma \ref{lem:HWPDstep1} with $\ell=1$, the leftmost cross in each row has a distinct column index after applying step (1) of the algorithm, and thus we can freely move each of these crosses down within their own columns to apply step (2).  In addition, since the shifted column sequence $(j_1^i+i-1)$ is strictly increasing by Lemma \ref{lem:HWPDstep1}, and these crosses are moved down within the same column to row $(j_1^i+i-1)$, then the relative row ordering of crosses with $\ell=1$ is preserved by the algorithm.
   
   In step (3), we must show that it is always possible to move the cross under consideration down within the same column, subject to the three criteria articulated in step (3a). Let $k$ enumerate those crosses that are fixed in steps (3) and (4) of the algorithm, in the order in which they get fixed; that is, $k$ enumerates from the bottom up those crosses having iteration index $\ell \geq 2$. We proceed by induction on $k$ to prove that the $k^{\text{th}}$ cross $c_k$ can be placed in a manner consistent with properties (i) through (iii). We simultaneously argue by induction on $k$ that $\row(\tilde{c}_{k+1}) < \row(\tilde{c}_k)$ for all $k \in \N$ such that $c_k$ and $c_{k+1}$ have the same iteration index; equivalently by Lemma \ref{lem:HWPD}, we show that the algorithm preserves the relative row ordering of crosses with the same iteration index. 

    The base case for placing the cross $c_1$ according to the algorithm concerns the lowest cross which is second in its original row in $D$. By Lemma \ref{lem:HWPDstep1}, the column sequence for those crosses previously fixed in step (2) is strictly increasing. Since in step (2) these crosses move downward such that their row and column index match, then any cross below $c_1$ and in the same column must lie in rows strictly below the fixed cross $\tilde{c}'_1$. (For example, $c_1=(3,5)$ in Figure \ref{fig:crossesShifted}. The fixed cross below $c_1$ at position $(5,5)$ lies in a row strictly below the fixed cross $\tilde{c}'_1 = (3,3)$.) In particular, $c_1$ can move freely downward within its column, at least as far down as to rejoin $\tilde{c}'_1$ in order to become the second cross in its new row (for example, $c_1 = (3,5)$ moves to $\tilde{c}_1 = (4,5)$ in step (3a) in Figure \ref{fig:crossesShifted}, which is in a row below $\tilde{c}'_1 = (3,3)$). Moving $c_1$ in this manner preserves properties (i), (ii), and (iii), and thus verifies the base case for this claim.

    We now verify the base case $\row(\tilde{c}_2) < \row(\tilde{c}_1)$ for the second claim. Recall that $c_1$ has already been fixed to obtain $\tilde{c}_1$. By the same argument, any cross below $c_2$ and in the same column must lie in a row strictly below the fixed cross $\tilde{c}'_2$, so that $c_2$ can move freely down its column, at least as far down as to join $\tilde{c}'_2$ to become the second cross in its new row. Moving $c_2$ in this manner clearly satisfies (i) and (ii). Denote by $\row(\tilde{c}_2) = m$ and $\row(\tilde{c}_1) = q$, in which case there are crosses on the diagonal in positions $(m,m)$ and $(q,q)$ from step (2) and property (ii). If $q< m$, then in fact $c_1$ could have moved farther down into row $m$ without moving through any crosses, based on the relative row and column ordering of crosses with iteration index $\ell=1$, and the fact that $c_1$ is the lowest cross with index $\ell=2$. Since $\tilde{c}_1$ moves as low as possible by the algorithm, then $\row(\tilde{c}_2) = m \leq q = \row(\tilde{c}_1)$. In order that $\tilde{c}_2$ satisfies (ii), then additionally $m \neq q$. Therefore, $\row(\tilde{c}_2) < \row(\tilde{c}_1)$, confirming the base case for this inequality.

     For the inductive step, assume that crosses $c_1$ through $c_k$ have been successfully placed by the algorithm. Further assume that $\row(\tilde{c}_j) < \row(\tilde{c}_{j-1})$ for all $2 \leq j \leq k$ such that $c_{j-1}$ and $c_j$ have the same iteration index. There are two cases to consider in the inductive step: either crosses $c_k$ and $c_{k+1}$ have the same iteration index or they do not. If $c_k$ has iteration index $\ell-1$ and $c_{k+1}$ has iteration index $\ell$, then the same argument supplied in the base case, instead applying Lemma \ref{lem:HWPDstep1} to the $(\ell-1)^{\text{st}}$ column sequence, shows that the first cross having iteration index $\ell$, namely $c_{k+1}$, can be placed in a manner satisfying (i), (ii), and (iii). The claim concerning the row indices does not apply in case $c_k$ and $c_{k+1}$ have different iteration indices.

     Now suppose that crosses $c_k$ and $c_{k+1}$ have the same iteration index $\ell \geq 2$. Recall that $c'_k$ and $c'_{k+1}$ are the $(\ell-1)^{\text{st}}$ crosses in the original rows of $D$ with $c_k$ and $c_{k+1}$, respectively. The same argument supplied in the base case showing that $\row(\tilde{c}_2) < \row(\tilde{c}_1)$, here instead applying Lemma \ref{lem:HWPDstep1} to the $(\ell-1)^{\text{st}}$ column sequence, shows that $c_k$ and $c_{k+1}$ can move freely down their columns, at least as far down as to rejoin $\tilde{c}'_k$ and $\tilde{c}'_{k+1}$, respectively. Note that this already implies that $c_{k+1}$ can move downward in a manner preserving (i) and (ii). Moreover, the inductive hypothesis implies that $\row(\tilde{c}'_{k+1}) < \row(\tilde{c}'_k)$, and so the same argument as in the base case proves that $\row(\tilde{c}_{k+1}) < \row(\tilde{c}_k)$, verifying the claimed inequality on row indices.  Therefore, $c_{k+1}$ can indeed be placed in a manner respecting criteria (iii) of the algorithm.

    Altogether, we have thus shown that all crosses of $D$ can be placed in a manner which satisfies criteria (i), (ii), and (iii) from step (3a) of the algorithm, provided that we enumerate the crosses beginning at the bottom row, ordered by iteration index $\ell$ as specified in step (4).  Therefore, Algorithm \ref{alg:EG} indeed produces a well-defined diagram $\widetilde{D}$.

    Finally, recall that crosses which are placed in step $\ell$ of the algorithm correspond to the $\ell^{\text{th}}$ cross in their original row in $D$. In steps (1) and (2), first crosses in rows of $D$ become first crosses in each row of $\widetilde{D}$.  By criterion (ii) of step (3a), crosses that were in position $\ell \geq 2$ remain the $\ell^{\text{th}}$ crosses from the left within their new rows in $\widetilde{D}$. Therefore, the weight of $\widetilde{D}$ is simply a permutation of the weight of the original pipe dream $D$.
\end{proof}

The truncating permutation $\pi_D$ from Theorem \ref{thm:main} is obtained from the diagram $\widetilde{D}$ as follows.

\begin{theorem}\label{thm:AlgCor}
Let $D \in RP(w)$ be a highest weight pipe dream for $w \in S_n$. Then $\pi_D \in S_n$ from Theorem \ref{thm:main} is the unique shortest permutation such that $\operatorname{wt}(\widetilde{D}) = \pi_D(\wt(D))$.  In addition, the composition $a_D = \wt(\widetilde{D})$ in Theorem \ref{thm:main-key} indexes the key polynomial corresponding to the pair $(\pi_D, \wt(D))$. 
\end{theorem}

We can now extract the truncating permutation $\pi_D$ from Example \ref{ex:alg} via Theorem \ref{thm:AlgCor}.

\begin{example}
    The weight of the pipe dream $D$ from Example \ref{ex:alg} is $\wt(D) = (5,3,3,1,1,0)$. After applying Algorithm \ref{alg:EG}, we obtained the diagram $\widetilde{D}$ in Figure \ref{fig:ellIs5} such that $a_D = \wt(\widetilde{D}) = (3,5,1,3,1,0)$.  The shortest permutation $\pi_D$ such that $a_D = \pi_D(\wt(D))$ equals $\pi_D = s_1s_3$, since $(3,5,1,3,1,0) = s_1s_3(5,3,3,1,1,0)$, and no single transposition relates $\wt(\widetilde{D})$ and $\wt(D)$.
\end{example}

\subsection{Weak Edelman--Greene insertion}

To prove Theorem \ref{thm:AlgCor}, we relate $\pi_D$ with the sorting permutation $\pi_r$ from Theorem \ref{thm:main-rfc}. To obtain $\pi_r$ following \cite{assaf_demazure_2018}, we must apply a modification of Edelman--Greene insertion.  We thus continue by reviewing this insertion algorithm.

In \cite{edelman_balanced_1987}, Edelman and Greene provide a Schur expansion for Stanley symmetric functions via an insertion algorithm for adding the letters of an ordered alphabet to a tableau, as follows.

\begin{definition}
    \label{def:ed_insertion}
    Let $P$ be a Young tableau, let $P_i$ represent the $i^\text{th}$ row of $P$ from the bottom, and let $x$ be a letter in an ordered alphabet. The \emph{Edelman--Greene insertion of $x$ into $P$}, denoted $P \leftarrow x$, is the tableau created by letting $x_0 = x$, and applying the following algorithm, beginning with $i = 0$:

    \begin{enumerate}
        \item If $x_i \geq z$ for all $z \in P_{i+1}$, then append $x_i$ to $P_{i+1}$, and the algorithm terminates. 
        \item Otherwise, let $x_{i+1}$ be the smallest entry in $P_{i+1}$ such that $x_{i+1} > x_i$. Then:
        \begin{enumerate}
            \item If $x_{i+1} \neq x_i+1$ or $x_i \notin P_{i+1}$, then replace $x_{i+1}$ with $x_i$ in $P_{i+1}$. 
            \item Otherwise, leave $P_{i+1}$ unchanged.   
        \end{enumerate}
         \item Increment $i$ by one. 
         \begin{enumerate}
          \item If $P_{i+1}$ is empty, the algorithm terminates. 
          \item Otherwise, return to step (1).
          \end{enumerate}
    \end{enumerate}
\end{definition}

We illustrate Edelman--Greene insertion in the following example, constructed such that each case of the algorithm occurs exactly once.

\begin{figure}[h]
    \centering
    \def\t{7}
    \begin{tikzpicture}[scale=0.6]
        \draw (0,0) -- (3,0) -- (3,1) -- (0,1);
        \draw (2,0) -- (2,3);
        \draw (0,0) -- (0,4);
        \draw (1,0) -- (1,4);
        \draw (0,1) -- (2,1);
        \draw (0,2) -- (2,2);
        \draw (2,3) -- (0,3);
        \draw (0,4) -- (1,4);

        \node [align=center] at (-1, 2) {\large $P = $};
        \node [align=center] at (0.5, 0.5) {\Large 1};
        \node [align=center] at (1.5, 0.5) {\Large 2};
        \node [align=center] at (2.5, 0.5) {\Large 4};
        \node [align=center] at (0.5, 1.5) {\Large 3};
        \node [align=center] at (1.5, 1.5) {\Large 5};
        \node [align=center] at (0.5, 2.5) {\Large 5};
        \node [align=center] at (1.5, 2.5) {\Large 6};
        \node [align=center] at (0.5, 3.5) {\Large 6};

        \draw (0+\t,0) -- (3+\t,0) -- (3+\t,1) -- (0+\t,1);
        \draw (2+\t,0) -- (2+\t,3);
        \draw (0+\t,0) -- (0+\t,4);
        \draw (1+\t,0) -- (1+\t,4);
        \draw (0+\t,1) -- (2+\t,1);
        \draw (0+\t,2) -- (2+\t,2);
        \draw (2+\t,3) -- (0+\t,3);
        \draw (0+\t,4) -- (1+\t,4);
        \draw (1+\t, 4) -- (2+\t, 4) -- (2+\t, 3);

        \node [align=center] at (-1+\t-0.5, 2) {\large $P \leftarrow 2 = $};
        \node [align=center] at (0.5+\t, 0.5) {\Large 1};
        \node [align=center] at (1.5+\t, 0.5) {\Large 2};
        \node [align=center] at (2.5+\t, 0.5) {\Large 2};
        \node [align=center] at (0.5+\t, 1.5) {\Large 3};
        \node [align=center] at (1.5+\t, 1.5) {\Large 4};
        \node [align=center] at (0.5+\t, 2.5) {\Large 5};
        \node [align=center] at (1.5+\t, 2.5) {\Large 6};
        \node [align=center] at (0.5+\t, 3.5) {\Large 6};
        \node [align=center] at (1.5+\t, 3.5) {\Large 6};
    \end{tikzpicture}
    \caption{A Young tableau $P$ (left) and the result of inserting the letter $2$ (right).}
    \label{fig:insertion_ex}
\end{figure}

\begin{example}
    \label{ex:insertion}
    Consider the tableau $P$ as shown on the left in Figure \ref{fig:insertion_ex}. We will compute $P \leftarrow 2$, meaning that we will insert $2$ into $P$ via Definition \ref{def:ed_insertion}. 
    
    Begin by setting $i = 0$ and $x_0 = 2$. Step (1) asks if $2 = x_0 \geq z$ for all $z \in P_{i+1} = P_{0+1}$.  Since $2 \not\geq 4 \in P_1$, we proceed to step (2) and set $x_{i+1} = x_1 = 4$. Now observe that $4 = x_1 \neq x_0 + 1 = 3$, and so step (2a) says that we replace $x_1 = 4$ with $x_0 = 2$ in $P_1$. In step (3), we return to step (1) of the algorithm instead with $i = 1$.

    Now $i=1$ and $x_1 = 4$.  Since $4 \not\geq 5 \in P_2$, we proceed to step (2) with $x_{i+1}=x_2 = 5$. Although $x_2 = x_1 + 1$, we also see that $x_1 = 4 \notin P_2$. We thus replace $x_2=5$ with $x_1=4$ in $P_2$, and repeat the process from step (1) with $i=2$.

    Now $i=2$ and $x_2=5$. Since $5 \not\geq 6 \in P_3$, we proceed to step (2) with $x_3 = 6$. Here, both $x_3 = x_2 + 1$ and $x_2 = 5 \in P_3$, and so we are in case (2b). We thus leave $P_3$ unchanged, and repeat with $i=3$.

    Now $i=3$ and $x_3 = 6$. In this case, $x_3=6 \geq z$ for all $z \in P_4$, and so we simply append $6$ to the end of $P_4$ according to step (1).  Finally, since there are no boxes in $P_5$, the algorithm terminates in step (3a). The resulting tableau $P \leftarrow 2$ is displayed on the right in Figure \ref{fig:insertion_ex}.
\end{example}

In \cite[Section 4.2]{assaf_demazure_2018}, the Edelman--Greene insertion algorithm is applied iteratively by inserting the letters of a given reduced word.  Let  $\bfa = a_1 \cdots a_p \in R(w)$ for some $w \in S_n$. The \emph{insertion tableau for $\bfa$}, denoted $P(\bfa)$, is the tableau that results from Edelman--Greene inserting the letters $a_1, \dots, a_p$ one by one, starting with an empty tableau.
An example of an insertion tableau is given in Figure \ref{fig:insertion_tab}, demonstrating the letter-by-letter insertion of the word $\bfa = 23124$.

\begin{figure}[h]
    \centering
    \def\t{3.5}
    \begin{tikzpicture}[scale=0.6]
        \draw (0,1) -- (0,0) -- (1,0);

        \draw (\t, 0) -- (1+\t, 0) -- (1+\t, 1) -- (\t, 1) -- (\t, 0);
        \node [align=center] at (0.5+\t, 0.5) {\Large 2};

        \draw (2*\t, 0) -- (2*\t+2, 0) -- (2*\t+2, 1) -- (2*\t, 1) -- (2*\t, 0);
        \draw (2*\t+1, 0) -- (2*\t+1, 1);
        \node [align=center] at (0.5+2*\t, 0.5) {\Large 2};
        \node [align=center] at (1.5+2*\t, 0.5) {\Large 3};

        \draw (3*\t, 0) -- (3*\t+2, 0) -- (3*\t+2, 1) -- (3*\t, 1) -- (3*\t, 0);
        \draw (3*\t+1, 0) -- (3*\t+1, 1);
        \draw (3*\t, 1) -- (3*\t, 2) -- (3*\t+1, 2) -- (3*\t+1,1);
        \node [align=center] at (0.5+3*\t, 0.5) {\Large 1};
        \node [align=center] at (1.5+3*\t, 0.5) {\Large 3};
        \node [align=center] at (0.5+3*\t, 1.5) {\Large 2};

        \draw (4*\t, 0) -- (4*\t+2, 0) -- (4*\t+2, 1) -- (4*\t, 1) -- (4*\t, 0);
        \draw (4*\t+1, 0) -- (4*\t+1, 1);
        \draw (4*\t, 1) -- (4*\t, 2) -- (4*\t+1, 2) -- (4*\t+1,1);
        \draw (4*\t+1, 2) -- (4*\t+2, 2) -- (4*\t+2, 1);
        \node [align=center] at (0.5+4*\t, 0.5) {\Large 1};
        \node [align=center] at (1.5+4*\t, 0.5) {\Large 2};
        \node [align=center] at (0.5+4*\t, 1.5) {\Large 2};
        \node [align=center] at (1.5+4*\t, 1.5) {\Large 3};

        \draw (5*\t, 0) -- (5*\t+2, 0) -- (5*\t+2, 1) -- (5*\t, 1) -- (5*\t, 0);
        \draw (5*\t+1, 0) -- (5*\t+1, 1);
        \draw (5*\t, 1) -- (5*\t, 2) -- (5*\t+1, 2) -- (5*\t+1,1);
        \draw (5*\t+1, 2) -- (5*\t+2, 2) -- (5*\t+2, 1);
        \draw (5*\t+2, 0) -- (5*\t+3, 0) -- (5*\t+3, 1) -- (5*\t+2, 1);
        \node [align=center] at (0.5+5*\t, 0.5) {\Large 1};
        \node [align=center] at (1.5+5*\t, 0.5) {\Large 2};
        \node [align=center] at (0.5+5*\t, 1.5) {\Large 2};
        \node [align=center] at (1.5+5*\t, 1.5) {\Large 3};
        \node [align=center] at (2.5+5*\t, 0.5) {\Large 4};
    \end{tikzpicture}
    \caption{The insertion tableau $P(23124)$, with the result after each insertion.} 
    \label{fig:insertion_tab}
\end{figure}

We now consider the lift operation on Young tableaux, originally defined by Assaf in \cite[Section 5.2]{assaf-EG}. Here, the operator produces tableaux of ``key shape'' in the sense of \cite[Definition 3.1]{assaf_demazure_2018}, generalizing \cite[Definition 2.5]{assaf-weak_dual_equiv}.

\begin{definition}[{Section 5.2 \cite{assaf_demazure_2018}}]
    \label{def:lift}
    If $P$ is a Young tableau with rows enumerated from bottom to top, produce its \emph{lift}, denoted $\lift(P)$, as follows. 
    \begin{enumerate}
        \item Position each entry of the leftmost column in a row whose index matches the entry.
        \item For each subsequent column, raise the entries from top to bottom of the column, maintaining their relative order, to the highest available row such that there exists an entry to its immediate left which is is strictly smaller.
    \end{enumerate}
\end{definition}

\begin{figure}[h]
    \centering
    \def\t{7}
    \begin{tikzpicture}[scale=0.6]
        \draw (0,0) -- (3,0) -- (3,1) -- (0,1);
        \draw (0,2) -- (2,2) -- (2,0);
        \draw (0,3) -- (1,3) -- (1,0);
        \draw (0,0) -- (0,3);

        \draw (\t, 0) -- (\t, 4);
        \draw (\t,0) -- (\t+1, 0) -- (\t+1, 2);
        \draw (\t, 1) -- (\t+3,1) -- (\t+3, 2) -- (\t+0,2);
        \draw (\t+2,1) -- (\t+2,2);
        \draw (\t, 3) -- (\t+2, 3) -- (\t+2, 4) -- (\t, 4);
        \draw (\t+1, 3) -- (\t+1, 4);

        \node [align=center] at (0.5, 0.5) {\Large $1$};
        \node [align=center] at (1.5, 0.5) {\Large $3$};
        \node [align=center] at (2.5, 0.5) {\Large $4$};
        \node [align=center] at (0.5, 1.5) {\Large $2$};
        \node [align=center] at (1.5, 1.5) {\Large $5$};
        \node [align=center] at (0.5, 2.5) {\Large $4$};

        \node [align=center] at (\t+0.5, 0.5) {\Large $1$};
        \node [align=center] at (\t+0.5, 1.5) {\Large $2$};
        \node [align=center] at (\t+1.5, 1.5) {\Large $3$};
        \node [align=center] at (\t+2.5, 1.5) {\Large $4$};
        \node [align=center] at (\t+0.5, 3.5) {\Large $4$};
        \node [align=center] at (\t+1.5, 3.5) {\Large $5$};

        \node [align=center] at (-1,1.5) {\large $P = $};
        \node [align=center] at (-1.5+\t,1.5) {\large $\lift(P) = $};
    \end{tikzpicture}
    \caption{A Young tableau $P$ (left) and the key tableau given by its lift (right).}
    \label{fig:lift_ex}
\end{figure}

\begin{example}
    \label{ex:lift}
    Consider $P$ as depicted on the left in Figure \ref{fig:lift_ex}.  To compute $\lift(P)$, we start by lifting the 4 in the first column up to row 4 as required by step (1). Next, we lift column 2. We first lift $5$ to row $4$, since this is the highest row with an entry strictly less than $5$ in the first column. Similarly, $3$ then gets lifted to row $2$. Finally in column $3$, we lift $4$ to row $2$ since this is the highest row with an entry less than 4 to its immediate left.  The outcome $\lift(P)$ is shown on the right in Figure \ref{fig:lift_ex}.
\end{example}

We now combine the Edelman--Greene insertion process with the lift operator. Given any reduced word $\bfa \in R(w)$, define the \emph{weak insertion tableau} to be $\lift(P(\bfa))$. Assaf refers to calculating $\lift(P(\bfa))$ as the \emph{weak Edelman--Greene correspondence} \cite{assaf-EG}. Denote by $\wt(\lift(P(\bfa)))$ the weak composition whose $i^{\text{th}}$ coordinate equals the number of boxes in row $i$ of $\lift(P(\bfa))$, enumerated from bottom to top.  

\begin{definition}\label{def:pi_r}
    Given any $w \in S_n$ and $r \in RFC(w)$, denote by $\pi_r = [\pi_1 \cdots \pi_n]$ the shortest permutation that sorts the weak composition $a = \wt(\lift(P(r))$, in the sense that $(a_{\pi_1}, \dots, a_{\pi_n})$ is a partition.
\end{definition}

Assaf and Schilling prove that $\pi_r$ is the truncating permutation for the Demazure crystal corresponding to any highest weight factorization $r$; see Theorem \ref{thm:main-rfc}.  Equivalently, the composition $\wt(\lift(P(r))$ indexes a key polynomial in \cite[Corollary 5.12]{assaf_demazure_2018}, denoted there instead by $\sh(\widehat{P}(r))$. Our remaining goal is thus to explicitly relate $\pi_D$ from Theorem \ref{thm:main} with this permutation $\pi_r$.

\subsection{Proof of Theorem \ref{thm:AlgCor}}

This final subsection is dedicated to the proof of Theorem \ref{thm:AlgCor}, which then completes Theorem \ref{thm:main}.  We begin with several preparatory lemmas, and our final proposition identifies the insertion tableau obtained by applying Edelman--Greene insertion to any highest weight factorization.

We first show that each sequence of corresponding letters from blocks of a highest weight factorization strictly decreases.

\begin{lemma}\label{lem:increase}
  Let $r \in RFC(w)$ for $w \in S_n$ be a highest weight factorization. If $r^i_j$ denotes the $j^{\text{th}}$ element in block $(r^i)$, then $r_j^{n-1} > \cdots > r_j^2 > r_j^1$, where we omit letters from any empty blocks.
\end{lemma}

\begin{proof}
This follows directly from Lemma \ref{lem:HWPDstep1} and the definition of $\Phi$ from Proposition \ref{prop:wtBijPD}.
\end{proof}

We now run the Edelman--Greene insertion algorithm on a highest weight factorization.

\begin{example}\label{ex:P(r)figs}
    The three highest weight factorizations for $w = [21543] \in S_5$ are $(\;)(4)(3)(14),$ $(\;)(\;)(34)(13),$ and $(\;)(\;)(4)(134)$. Applying the Edelman--Greene algorithm to each of these factorizations produces the insertion tableaux pictured in Figure \ref{fig:highestPDandTab}, recorded from left to right.  Note that the result of the algorithm simply places the entries of block $(r^i)$ in row $i$ of the tableau.
\end{example}

\begin{figure}[h]
    \centering
    \includegraphics{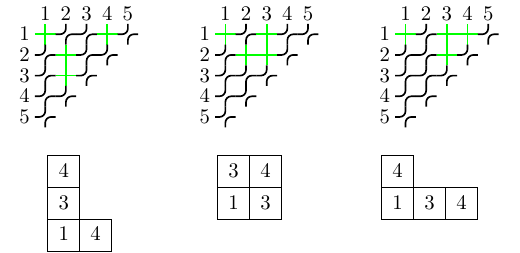}
    \caption{Insertion tableaux $P(r)$ and pipe dreams $\Phi^{-1}(r)$ for the highest weight factorizations $r \in \{(\;)(4)(3)(14), (\;)(\;)(34)(13),(\;)(\;)(4)(134)\} \subset RFC([21543])$.}
    \label{fig:highestPDandTab}
\end{figure}

The following proposition generalizes the observation recorded in Example \ref{ex:P(r)figs}.

\begin{proposition}\label{prop:highestWtSh}
    Let $r \in RFC(w^{-1})$ for some $w \in S_n$.  If $r$ is a highest weight factorization, then the insertion tableau $P(r)$ is given by placing the entries in block $(r^i)$ of $r$ in row $i$ of the tableau.
\end{proposition}

\begin{proof}
     Let $r \in RFC(w^{-1})$, and suppose that $r$ is a highest weight factorization.  Our general strategy will be to show that when we run the Edelman--Greene insertion algorithm to obtain the insertion tableau $P(r)$, inserting the letter $r_j^i$ replaces $r_j^{i+1}$ via step (2a) by bumping all $r_j^{\ell}$ for $\ell \geq i+1$ up one row within the same column. 
     
     We proceed first by induction on the number of nonempty blocks in $r$. Suppose that $r = (\;)\cdots (\;)(r^1)$ has only one nonempty block, having entries $r_1^1 < \cdots < r^1_m$. Then case (1) applies in every iteration of the Edelman--Greene insertion algorithm, and the entries of block $(r^1)$ are simply placed in row 1 of the tableau, verifying the base case.

    Now suppose that $r$ has at least two nonempty blocks, and proceed to the induction step. Let $r' = (r^{n-1})(\cdots)(r^{i+1})$, so that the insertion tableau $P(r')$ is obtained by placing the entries of blocks $(r^{i+1}), \dots, (r^{n-1})$ in rows $1, \dots, n-i-1$ of the tableau $P(r')$, by induction on the number of blocks. We must apply the Edelman--Greene insertion algorithm using the letters of block $(r^i)$.

     We now further proceed by induction on the length of the block. Suppose that $(r^i) = (r^i_1)$ has one letter. Since $r$ is a highest weight factorization, then the weight of $\Phi^{-1}(r) = D$ is a partition by Corollary \ref{cor:intertwine-e} and Lemma \ref{lem:HWPD}. Since $\Phi$ is weight-preserving, all blocks $(r^{i+1}), \dots, (r^{n-1})$ also have at most one letter. 
    
    We insert the letter $r^i_1$ into  $P(r')$, containing the entries of blocks $(r^{i+1}), \dots, (r^{n-1})$ in rows $1, \dots, n-i-1$, respectively.  Since the letters within each block are strictly increasing and we have shown in Lemma \ref{lem:increase} that $r_1^i < r_1^{i+1}$, then $r_1^i \notin P_1(r')$ and case (2a) of the insertion algorithm applies.  We thus replace $r^{i+1}_1$ by $r_1^i$ in row 1. Iterating using the fact that blocks $(r^{i+1}), \dots, (r^{n-1})$ also have at most one letter, we see that $r^\ell_1$ simply bumps $r^{\ell+1}_1$ up one row in the same column for all $\ell \geq i+1$. As there are no additional letters in any of the last $n-i-1$ blocks of $r$, the algorithm terminates and the result holds in the case where $(r^i)$ has one letter.

    Now suppose that $(r^i)$ has at least two letters and that we are inserting the letter $r^i_j$ for $j \geq 2$ into the first row. Denote the length of block $(r^{i+1})$ by $k$, and first suppose that $j \leq k$.  By induction on the length of $(r^i)$, say $m$, we know that letter $r^i_{j-1}$ has bumped $r^{i+1}_{j-1}$ up one row in the same column.  In particular, the first row of the current tableau is filled by the initial subsequence $r^i_1, \dots, r^i_{j-1}, r^{i+1}_j$. To insert $r^i_j$, recall that $r^i_{j-1} < r^i_j$ since the entries of $r$ increase in blocks. In addition, recall by Lemma \ref{lem:increase} that $r^i_j < r^{i+1}_j$. Therefore, $r_j^i$ is not in the first row filled by the initial subsequence $r^i_1, \dots, r^i_{j-1}, r^{i+1}_j$.  Thus, the letter $r_j^i$ indeed replaces $r^{i+1}_j$ via case (2a) of the insertion algorithm if $j \leq k$. By the inductive hypothesis, the second row now begins with the initial subsequence $r^{i+1}_1, \dots, r^{i+1}_{j-1}, r^{i+1}_j$, provided that $j \leq k$. 

    For all $r^i_j$ such that $j > k$, namely $r^i_{k+1} < \cdots < r^i_m$, since all letters of block $(r^{i+1})$ have already been bumped to the second row and letters strictly increase within block $(r^i)$, then case (1) applies. These remaining letters are then appended to the end of the bottom row by the insertion algorithm. Altogether, we have proved that the entries in block $(r^i)$ are simply placed in the first row of the tableau when inserting the letters from block $i$.  The result now follows by induction.
\end{proof}

We are now prepared to prove Theorem \ref{thm:AlgCor}, which says that $\operatorname{wt}(\widetilde{D}) = \pi_D(\wt(D))$ for any highest weight pipe dream $D \in RP(w)$, where $\pi_D$ is the truncating permutation from Theorem \ref{thm:main}.

\begin{proof}[Proof of Theorem \ref{thm:AlgCor}]
    Let $D \in RP(w)$ be a highest weight pipe dream for $w \in S_n$. Denote by $\Phi$ the bijection from Proposition~\ref{prop:wtBijPD}, so that $\Phi(D) = r \in RFC(w^{-1})$ is a highest weight factorization by Corollary \ref{cor:intertwine-e}. By Theorem \ref{thm:main-rfc}, the truncating permutation $\pi_r = [\pi_1 \cdots \pi_n] \in S_n$ for the Demazure crystal $B_{\pi_r}(\wt(r))$ is the shortest permutation that sorts the weak composition $a = \wt(\lift(P(r))$, in the sense that $(a_{\pi_1}, \dots, a_{\pi_n})$ is a partition. By Proposition \ref{prop:highestWtSh}, however, since $r$ is a highest weight factorization, then $\wt(P(r))$ is a partition. Moreover, since by Proposition \ref{prop:highestWtSh} the tableau $P(r)$ records the entries of block $(r^i)$ in row $i$, then $\wt(P(r)) = \wt(r)$.  Therefore, Theorem \ref{thm:main-rfc} says that $\wt(r) = (a_{\pi_1}, \dots, a_{\pi_n})$, or equivalently,
    \begin{equation}\label{eq:lifteq}
        \wt(r) = \pi_r^{-1} a = \pi_r^{-1} \left( \wt(\lift(P(r)))\right)
    \end{equation}
    because the left action of $S_n$ on $\Z^n$ permutes coordinates.

    Since $\Phi$ is weight-preserving, then $\wt(r) = \wt(\Phi(D)) = \wt(D)$. 
    We claim that $\wt(\lift(P(r))) = \wt(\widetilde{D})$. 
    To prove this, we show that Algorithm \ref{alg:EG} acts on the rows of $D$ in the same way that the lift operator acts on the rows of $P(r)$. By definition of $\Phi$, step (1) of Algorithm \ref{alg:EG} recovers the entries of $r$ by row from top to bottom, reading column indices. Proposition \ref{prop:highestWtSh} equates this output with the insertion tableau $P(r)$. Next observe that step (2) of Algorithm~\ref{alg:EG} corresponds to step (1) of the lift operator from Definition~\ref{def:lift}, noting that the leftmost column of $\lift(P(r))$ corresponds via $\Phi$ to the leftmost cross in each row of the shifted version of $D$. In step (3a) of Algorithm~\ref{alg:EG}, condition (i) corresponds to maintaining the relative order of entries, condition (ii) corresponds to the existence of an entry on the immediate left that is smaller, and condition (iii) reflects that the lift operates column by column. As the lift operator moves entries one at a time within a column, steps (3b) and (3c) of Algorithm \ref{alg:EG} move crosses of $\widetilde{D}$ in the same manner. Incrementing $\ell$ in step (4) of Algorithm \ref{alg:EG} ensures that step (3), equivalently the lift operation, is applied to each column from left to right on $P(r)$. Therefore, $\wt(\lift(P(r))) = \wt(\widetilde{D})$, and so \eqref{eq:lifteq} becomes
    \begin{equation}\label{eq:perm}
        \wt(D) = \wt(r) = \pi_r^{-1} \left( \wt(\lift(P(r)))\right) = \pi_r^{-1} (\wt(\widetilde{D})).
    \end{equation}

    Finally, recall from the proof of Theorem \ref{thm:main} that the bijection $\Phi : RP(w) \to RFC(w^{-1})$ intertwines with all raising and lowering crystal chute moves as well as the weight function, and thus preserves the Demazure crystal structure.  Since $\pi_r$ is the truncating permutation for the Demazure crystal $B_{\pi_r}(\wt(r))$, then $\pi_r$ is also the truncating permutation for the Demazure crystal denoted in Theorem \ref{thm:main} by $B_{\pi_D}(\wt(D))$, where $D = \Phi^{-1}(r)$. Therefore, rearranging \eqref{eq:perm}, we have $\wt(\widetilde{D}) = \pi_r (\wt(D)) = \pi_D (\wt(D)),$ completing the proof of Theorem \ref{thm:AlgCor}, and thus also Theorem \ref{thm:main-key} and Theorem \ref{thm:main}.
    \end{proof}

\bibliography{references}
\bibliographystyle{alphanum}

\end{document}